%% file: Paper.tex
\documentclass[10pt, reqno]{amsart}

\usepackage{amsmath,amssymb,amsthm,mathrsfs,amsfonts,dsfont,stmaryrd} 
\usepackage[backref]{hyperref}
\usepackage{multirow}
\usepackage{comment}
\usepackage{amscd}   
\usepackage{fullpage}
\usepackage[all]{xy} 
\usepackage{MnSymbol,graphicx}
\usepackage{mathrsfs,enumitem}
\usepackage[titletoc,toc,title]{appendix}
\usepackage{lscape}
\usepackage{pdflscape}

\usepackage{mathtools}
\usepackage[]{algorithm2e,algorithmic}

\usepackage[usenames,dvipsnames]{color}

\usepackage{tikz-cd} 
\usepackage{tikz}
\usepackage{array}

\newtheorem{lemma}{Lemma}[section]
\newtheorem{theorem}[lemma]{Theorem}
\newtheorem{proposition}[lemma]{Proposition}
\newtheorem{prop}[lemma]{Proposition}
\newtheorem{cor}[lemma]{Corollary}

\newtheorem{claim*}{Claim}

\newtheorem{corollary}[lemma]{Corollary}

\theoremstyle{definition}
\newtheorem{remark}[lemma]{Remark}

\newtheorem{example}[lemma]{Example}

\newtheorem{problem}[lemma]{Problem}
\newtheorem{defn}[lemma]{Definition}

\def\O{\mathcal{O}}

\newcommand{\C}{{\mathbb C}}
\newcommand{\F}{{\mathbb F}}

\newcommand{\Q}{{\mathbb Q}}

\newcommand{\Z}{{\mathbb Z}}

\newcommand{\calB}{{\mathcal B}}

\newcommand{\calM}{{\mathcal M}}

\newcommand{\calR}{{\mathcal R}}

\newcommand{\OO}{{\mathcal O}}

\newcommand{\frakp}{{\mathfrak p}}

\newcommand{\homEA}[1]{\fbox{\raisebox{0ex}[4ex][2ex]{$#1$}}}
\newcommand{\homAE}[1]{\framebox[4em][c]{$#1$}}
\newcommand{\homAA}[1]{\framebox[4em][c]{\raisebox{0ex}[4ex][2ex]{$#1$}}}
\newcommand{\EndEA}[4]{{\setlength\arraycolsep{0pt} \left(\begin{array}{cc} #1 & \homAE{#2} \\ \homEA{#3} & \homAA{#4}\end{array}\right)}}
\newcommand{\homEEA}[2]{{\setlength\arraycolsep{0pt} \left(\begin{array}{cc} \homEA{#1} &\homEA{#2}\end{array}\right)}}
\newcommand{\homEEEtoEA}[6]{{\setlength\arraycolsep{0pt} \left(\begin{array}{ccc}
#1 & #2 & #3 \\
\homEA{#4} &\homEA{#5} &\homEA{#6} \end{array}\right)}}

\newcommand{\xyzw}{\EndEA{x}{y}{z}{w}}

\newcommand{\polmatrix}{\EndEA{1}{0}{0}{\lambda_A}}

\DeclareMathOperator{\D}{D}

\DeclareMathOperator{\I}{I}

\DeclareMathOperator{\Tr}{Tr}

\DeclareMathOperator{\End}{End}

\DeclareMathOperator{\Hom}{Hom}

\DeclareMathOperator{\Aut}{Aut}
\DeclareMathOperator{\Gal}{Gal}

\DeclareMathOperator{\disc}{disc}

\DeclareMathOperator{\N}{N}

\newcommand{\Int}{\text{normal}}

\newcommand{\isom}{\cong}

\numberwithin{equation}{section}
\numberwithin{table}{section}

\author{Sorina Ionica}
\address{Sorina Ionica: Laboratoire MIS, Universit\'e de Picardie Jules Verne, 33 Rue Saint Leu, Amiens 80039, France}
\email{sorina.ionica@u-picardie.fr}

\author[K{\i}l{\i}\c{c}er]{P{\i}nar K{\i}l{\i}\c{c}er}

\address{%
	P{\i}nar K{\i}l{\i}\c{c}er:
Bernoulli Institute for Mathematics, Computer 
  Science and Artificial Intelligence, 9747 AG Groningen, Netherlands
}

\email{p.kilicer@rug.nl}
\author{Kristin Lauter}
\address{Kristin Lauter: Facebook AI Research, Meta, Seattle, WA, USA}
\email{klauter@fb.com}
\author{Elisa Lorenzo Garc\'ia}
 \address{Elisa Lorenzo Garc\'ia: Institut de Math\'ematiques, Universit\'e de Neuch\^atel, Rue Emile-Argand 11, 2000, Neuch\^atel, Switzerland --- Laboratoire IRMAR, Office 602,
Universit\'e de Rennes 1, Campus de Beaulieu, 35042, Rennes Cedex, France} \email{elisa.lorenzo@unine.ch, elisa.lorenzogarcia@univ-rennes1.fr}
\author{Adelina M\^{a}nz\u{a}\c{t}eanu}
\address{Adelina M\^{a}nz\u{a}\c{t}eanu: Mathematisch Instituut, Niels Bohrweg 1, 2333 CA Leiden, Netherlands}
\email{m.manzateanu@math.leidenuniv.nl}
\author{Christelle Vincent}
\address{Christelle Vincent: Department of Mathematics and Statistics, University of Vermont, 16 Colchester Avenue, Burlington VT 05401}
\email{christelle.vincent@uvm.edu}

\title{Determining the primes of bad reduction of CM curves of genus~$3$}

\begin{document}
\begin{abstract} In this paper we introduce a new problem called the Isogenous Embedding Problem (IEP). The existence of solutions to this problem is related to the primes of bad reduction of CM curves of genus $3$ and we can detect potentially good reduction in absence of solutions.  
We propose an algorithm for computing the solutions to the IEP and run the algorithm through different families of curves. We were able to prove the reduction type of some particular curves at certain primes that were open cases in \cite{LLLR}. 
\end{abstract}

\maketitle
\setcounter{tocdepth}{1} 
\tableofcontents

\section{Introduction}

\input{Sec1-Introduction}

\section{Bad reduction and the embedding problem} 
\input{Sec2-TheEmbeddingProblem}

\section{The embedding of the sextic order}\label{sec: embedding_img}  
\input{Sec3-Embedding}

\subsection{The embedding of the totally real cubic order}  
\label{sec: embedding_cubic}
\input{Sec3-1-EmbeddingOfTheCubic}

\subsection{The embedding of $\eta$} \label{sec: embedding_eta}  
\label{Sec4-EmbeddingOfO_eta}
\input{Sec3-2-EmbeddingOfO_eta}

\subsection{Sextic CM fields with imaginary quadratic subfields}\label{sec:iqf}  
\input{Sec4-SexticWithImaginaryQuadSubfield}

\section{Equivalence of solutions}\label{sec:eq} 
\input{Sec5-Equivalences}

\section{Counting the number of solutions}\label{sec:alg}  
\input{Sec6-CountingSolutions}

\input{Sec6.1-Bounds}

\input{Sec6.2-BoundsForNorms}
\input{Sec6.3-BoundsForPrimes}

\input{Sec6.5-Lifting}
\input{Sec6.6-Algorithm}

\section{Examples}\label{sec:comp} 
\input{Sec7-Examples}

\bibliographystyle{abbrv}
\bibliography{sample}

\end{document}

%% file: Sec1-Introduction.tex
Let $\mathcal{O}$ be an order in an imaginary quadratic field $K$. The \emph{class polynomial} associated to~$\mathcal{O}$ is defined to be 
$$H_\mathcal{O}(t)=\prod_{\{E/\mathbb{C}\text{ has CM by }\mathcal{O}\}/\simeq}(t-j(E)),$$
that is, the polynomial whose roots are the $j$-invariants of elliptic curves $E$ with complex multiplication by~$\mathcal{O}$. It is an irreducible polynomial with integer coefficients. Its roots generate a class field of $K$ and are used for constructing elliptic curves with a given number of points, which are useful in cryptography \cite{BrezingWeng, Barretto, Morain}. When $\mathcal{O}=\mathcal{O}_K$ is the maximal order of $K$, the class polynomial $H_{\mathcal{O}_K}(t)$ is called the 
\emph{Hilbert class polynomial} and its roots generate the Hilbert class field of $K$. 

In order to compute the class polynomial $H_\mathcal{O}(t)$, we begin by computing the set of fundamental $\tau\in\mathbb{H}$ that correspond to elliptic curves with complex multiplication (CM) by $\O$. These values $\tau \in \mathbb{H}$ can be obtained by computing the root with positive imaginary part of each primitive quadratic form of discriminant $\text{disc}(\mathcal{O})$ up to equivalence. The elliptic curve corresponding to $\tau$ is obtained by evaluating the $j$-invariant as a modular function at $\tau$ using its $q$-expansion, which gives an approximate value for the $j$-invariants, and hence the roots of the class polynomial. Since the coefficients of $H_\mathcal{O}(t)$ are integers, by evaluating the $j$-invariants to enough precision, we can compute the exact values of the coefficients \cite{BBEL, Broker,DrewHilbert}. 

This construction, known as the \emph{CM Method}, can be generalized to curves of genus $g$ with CM by an order~$\mathcal{O}$ in a CM field $K$ of degree $2g$, that is, a totally imaginary quadratic extension of a totally real number field of degree $g$. The theory of Shimura and Taniyama \cite{Shimura-Taniyama} tells us how to compute the \emph{Riemann matrix} $\tau\in\mathbb{H}_g$ for each Jacobian with CM by $\mathcal{O}$ up to isomorphism. We then need curve invariants which are analogous to the $j$-invariant. These invariants need to be absolute (quotients of equal weight invariants having a power of the discriminant in the denominator so that they are always defined) and modular (given by Siegel modular functions so that we can approximate their values at a given $\tau$). We call these invariants \emph{good invariants}. Given such invariants (see \cite[Section 2]{Streng2014} for $g=2$; \cite[Section 5]{Lor20} and \cite{LR19} for $g=3$), we may define class polynomials and compute class fields of $K$ in a similar fashion. However, in contrast to the dimension~$1$ case,  the coefficients of these polynomials are not integers, but rational numbers. Therefore, efficient computational methods for numerically approximating these polynomials and proving their complexity rely on explicit bounds for the denominators of class polynomials.\footnote{By denominator of a polynomial $f\in \mathbb{Q}[X]$ we mean the smallest $D\in \mathbb{Z}$ such that $Df\in \mathbb{Z}[X]$.}

For $g = 2$, the denominators of class polynomials are products of powers of primes of geometric bad reduction of the corresponding curves. Thanks to the work of Bruinier, Goren, Lauter, Viray, and Yang \cite{BruinierYang, GorenLauter07, GL11, LV15b}, there are explicit formulas for computing these denominators. As a consequence, a bound on the runtime of the algorithm computing the class polynomials in terms of the discriminant of the field $K$ is given in \cite{Streng2014}.

In higher genus, several families of curves with CM have been computed: for genus~3 hyperelliptic curves \cite{BILV, Weng}; Picard curves \cite{AEPicard, KLS, KoikeWeng, LarioSomoza}; plane quartics \cite{KLLRSS}; superelliptic curves \cite{Somozathesis}. Each of these papers lists examples of CM curves, but for most of these examples only a heuristic argument is given that the curves have CM by the correct order. While some of them are proven correct by running an algorithm computing the endomorphism of the Jacobian \cite{CMSV}, obtaining bounds on the denominators of class polynomials would allow to prove the correctness of the output of these algorithms in general.

In this paper, we turn our attention to the case of genus 3.
First, we note genus 3 curves are either plane quartics or hyperelliptic curves. When all curves with CM by an order in a sextic CM field are plane quartics, or hyperelliptic curves, then one may define class polynomials over $\mathbb{Q}$ as in genus 2. For some rare sextic CM fields, called mixed in \cite{DIS}, both hyperelliptic curves and plane quartics with CM by that field exist. To work around the problem of handling two different sets of invariants for these curves, one may define class polynomials whose roots are invariants of CM points in the Galois orbit of a certain Galois group \cite{DI,DIS}. CM points on these orbits 
are either Jacobians of hyperelliptic curves or of plane quartics and a single set of invariants may be used. To mark the distinction, we will call the polynomials defined in this way  \textit{hyperelliptic class polynomials} and \textit{plane quartic class polynomials}.  

The primes appearing in the discriminant of a plane quartic, and consequently in the denominators of plane quartic class polynomials, 
are either primes of geometric bad reduction or of hyperelliptic reduction of the curve\cite{LLLR}. It is conjectured in \cite[Section 4]{KLLRSS} that the primes of hyperelliptic reduction cannot be bounded in terms of the discriminant of the order, which suggests that the CM method used for $g = 1,2$ cannot be directly generalized to plane quartics. The case of genus 3 hyperelliptic curves with CM is much simpler, since all primes appearing in the denominators are primes of bad reduction \cite{IKLLMMV_mod}. 
In \cite{BCLLMNO} and \cite{KLLNOS}, the authors formulate for genus 3 a problem analogous to the embedding problem studied by Goren  and Lauter for genus 2. This allows to establish  bounds on the primes of bad reduction,  but these bounds are unfortunately too large to be practical.

In this paper, we work towards giving a better bound.
Indeed, we give new better bounds in some particular cases (see Section \ref{Ss:BoundForp}). In the general case, our main contribution  is of a slightly different flavour: we give an algorithm  which given a sextic CM field outputs a precise short list of primes of bad reduction for the curves with CM by this field (see Section \ref{sec: MainAlgorithm}). 
One of the consequences of these results is that we are able to answer the question about the reduction type of certain plane quartics studied in \cite{KLLRSS, LLLR} (see Proposition \ref{prop:hyperred}).

\subsection{Outline}
In Section \ref{Sec: Embedding Problem}, we introduce the objects that allow us to detect bad reduction: the solutions to the REP (Reduction Embedding Problem) and to the IEP (Isogenous Embedding Problem). Those embedding problems are generalizations to the \textit{Embedding Problem} in genus 2 as defined in \cite{GorenLauter07}.  In Section~\ref{sec: embedding_img}, we give explicit equations  for the solutions to the IEP that will be used later not only to bound the primes of bad reduction, but also to determine them. In Section \ref{sec:eq}, we define a natural equivalence of solutions to the REP and the IEP. Finally, in Section \ref{sec:alg}, we present our algorithm for computing solutions to the IEP. 
We run the algorithm for some sextic CM fields.
Finally, in Section \ref{sec:comp}, we present the outputs, as well as the conclusions. In particular, we are able to determine the reduction type of some plane quartics at some primes for which this was an open problem in \cite[Table 4]{LLLR}.

\subsection{Notation and assumptions}\label{sec: Embd}

Throughout the paper, $K$ is a sextic CM field, that is, a totally imaginary quadratic extension of a totally real cubic field $K_+$. Thus, $K_+ = \Q(\mu)$ and $K = K_+(\eta)$, where $\eta$ is a totally imaginary element in $K$ and $\mu = -\eta^2$. We denote the maximal orders of $K$ and $K_+$ by $\O_K$ and $\O_{K_+}$, respectively. A generic order of $K$, i.e., a $\Z$-module of rank $6$ in $\O_K$, is denoted by $\mathcal{O}$ and we denote $\mathcal{O}_{+} :=\mathcal{O}\cap K_+$.

By a \textit{curve}, we mean a smooth, projective,
geometrically irreducible curve. We denote by~$X$ a curve of genus 3 with CM by $\OO$ and with absolutely simple Jacobian $J := J(X)$, and by $M$ a number field over which $X$ and its stable model are defined. 

We say that a curve $X$ of genus 3 defined over a 
field $M$ has \emph{complex multiplication} (CM) by $\OO\subset K$ 
if there exists an embedding 
$\iota: \OO \hookrightarrow \End(J_{\overline{M}})$ such that $\iota^{-1}(\End(J_{\overline{M}})) = \OO$, 
where $J$ is the Jacobian of~$X$. A \textit{CM type} $\Phi$ of $K$ is 
a set of $3$ non-conjugate complex embeddings 
$K\hookrightarrow \C$. Given a Jacobian   $J/M$ with CM by $K$, then there is a unique CM type~$\Phi$ of $K$ such that the analytic representation of $\End_0(J_{\overline{M}})$ is equivalent to the direct sum representation~$\oplus_{\phi\in \Phi}\phi$. 
We say that a CM type is \textit{primitive} if its restriction to any strict CM
subfield of $K$ is not a CM type. The CM type corresponding to $J$
is \textit{primitive} if and only if~$J$ is \textit{absolutely} simple.

We denote the endomorphism ring of an elliptic curve $E$ by $\mathcal{R}$ and we define $\calB := \mathcal{R}\otimes\Q$. In this paper, an elliptic curve is always defined over a finite field $\F_q$, where $q = p^r$ and $p$ is a prime number, so the endomorphism algebra $\calB$ is either an indefinite quaternion algebra or an imaginary quadratic field.
If $E$ is supersingular, then  $\mathcal{B}=\mathcal{B}_{p,\infty}$ is the indefinite quaternion algebra only ramified at $p$ and $\infty$, and $\mathcal{R}$ is a maximal order in $\mathcal{B}$.

\subsection*{Acknowledgements} The authors would like to thank the Lorentz Center in Leiden for hosting the Women in Numbers Europe 2 workshop in September 2016 and providing a productive and enjoyable environment for our initial work on this project. We are grateful to the organizers of WIN-E2, Irene Bouw, Rachel Newton and Ekin Ozman, for making this conference and this collaboration possible. We also thank Christophe Ritzenthaler for useful discussions. This project was supported by the Thomas Jefferson Fund of the French-American Cultural Exchange Foundation. In addition, the work of the fourth author has been partially funded by the \textit{Melodia} ANR-20-CE40-0013 project, and the work of the sixth author is supported by the National Science Foundation, DMS-1802323.

%% file: Sec2-TheEmbeddingProblem.tex

\label{Sec: Embedding Problem}
In this section, we define two different embedding problems and we explain their connection to the primes of (geometric) bad reduction for curves of genus 3 with CM. Recall that the embedding problem for genus $2$ introduced in \cite[Section 3]{GorenLauter07} is the main ingredient in \cite{GorenLauter07,LV15b} for determining and controlling the primes and their exponents in the denominators of the coefficients of the Igusa Class Polynomials. 

First, we review some results from \cite{BCLLMNO} and \cite{KLLNOS}. We begin with enumerating all reduction types of CM curves of genus~3 modulo a prime ideal of the ring of integers of its field of definition.

\begin{prop}\label{reductionResult}\cite[Corollary 4.3]{BCLLMNO}
Let $X$ be a genus 3 curve with CM by an order $\O$ in a sextic CM field~$K$. Let $\mathfrak{p}$ be a prime ideal of $\O_M$, where $M$ is a number field over which $X$ and its stable model are defined. Then, one of the following statements holds for the special fiber $X_{\mathfrak{p}}$ of the stable reduction of $X$:
\begin{enumerate}
\item (good reduction) $X_{\mathfrak{p}}$ is a smooth curve of genus 3, 
\item \label{case:prodell}  $X_{\mathfrak{p}}$ has three irreducible components of genus 1,
\item \label{case:notprodell} $X_{\mathfrak{p}}$ has an irreducible component of genus 1 and one of genus 2. 
\end{enumerate}
\end{prop}

If $X$ and $\mathfrak{p}$ are such that one of the two last cases of Proposition~\ref{reductionResult} holds, then we say that $X$ has \textit{bad reduction} at $\mathfrak{p}$. The reader should be aware that what we call bad reduction in this paper is usually called \emph{geometric bad reduction}. 

The following result provides a motivation for the first embedding problem that we consider, namely the Reduction Embedding Problem.

\begin{proposition} \label{prop: embeddingR}\cite[Theorem 4.5 and Proposition 4.8]{BCLLMNO}
Let $X$ be a curve of genus $3$ defined over a number field $L$ with CM by an order $\OO$ in a sextic CM field~$K$. Suppose that the Jacobian~$J$ of $X$ is absolutely simple. Let $M$ be a  finite extension of $L$ such that $J$ has good reduction everywhere and $X$ has the stable model over $\O_M$. Let  $\mathfrak{p}$ be a prime ideal in $\O_M$ lying above~$p \subset \Z$. 

If $X$ has bad reduction modulo $\mathfrak{p}$, then there exists an elliptic curve~$E$ defined over $\overline{\mathbb{F}}_p$, a principally polarized abelian surface $A$ defined over $\overline{\mathbb{F}}_p$ with a polarization $\lambda_A$, and a ring embedding
\begin{equation*}
\iota_0:\,\O \hookrightarrow \End(E\times A)
\end{equation*}
such that the Rosati involution coming from the polarization $1\times \lambda_A$ induces complex conjugation on $\OO$. Moreover, in this situation, $A\sim E^2$ and the embedding can be taken to be optimal, i.e. with $[\iota_0^{-1}(\operatorname{End}(E\times A)\cap (\iota_0(\mathcal{O})\otimes\mathbb{Q})):\mathcal{O}]$ a power of $p$.
\end{proposition}

\begin{proof}
The proof of this proposition is given in \cite{BCLLMNO}, but we sketch the argument here to set up the notation. Since the Jacobian $J$ of $X/L$ has complex multiplication, by Serre--Tate \cite{SerreTate}, there exists an extension $M$ of~$L$ such that $J$ has potentially good reduction everywhere. Let $\mathfrak{p}\mid p$ be a prime in $\O_M$ such that $X$ has bad reduction modulo~$\mathfrak{p}$ whereas $J$ has good reduction modulo~$\mathfrak{p}$. Since $X$ has bad reduction, the reduction $J_{\mathfrak{p}}$ is decomposable. 

By Proposition \ref{reductionResult}, there are two possibilities for how $J_{\mathfrak{p}}$ decomposes; it is either a product of an elliptic curve and an abelian surface or a product of three elliptic curves. In both cases, one may assume that  $J_{\mathfrak{p}}\cong E \times A$ as principally polarized abelian varieties, where $E$ is an elliptic curve and $A$ is an abelian surface with a fixed principal polarization $\lambda_A$. We note that this includes the case where $A\isom E_1\times E_2$ with the product polarization. 

The embedding $\End(J)\hookrightarrow\End(J_{\mathfrak{p}})$ gives an embedding $\O \hookrightarrow \End(E\times A)$ and \cite[Proposition 4.8]{BCLLMNO} implies the condition on the Rosati involution. Moreover, it follows from \cite[Theorem 4.5]{BCLLMNO} that $A\sim E^2$ and from \cite[Corollary 6.1.2]{GL11} that the embedding is optimal.

\end{proof}

\begin{remark}\label{rmk:2cases} Notice that a supersingular abelian surface over $\overline{\mathbb{F}}_p$ is always isomorphic to $E^2$ (case of $a$-number $0$, see \cite{Oort1974}) or to $E^2/\alpha$ (case of $a$-number $1$, see \cite{Oort1975}) where $E$ is any supersingular elliptic curve over $\overline{\mathbb{F}}_p$ and $\alpha$ is an $\alpha$-group.
\end{remark}

\begin{proposition}\label{prop:ss} 
Let $X$ be a curve of genus $3$ defined over a number field $L$ with CM by an order $\OO$ in a sextic CM field~$K$. Suppose that the Jacobian~$J$ of $X$ is absolutely simple. 
Let~$M$ be a  finite extension of $L$ such that the $J$ has good reduction everywhere and $X$ has a stable model defined over $\O_M$. 
Let $\mathfrak{p} \nmid 2$ be a prime ideal of $\O_M$ such that $X$ has bad reduction at $\mathfrak{p}$. Then $J_{\mathfrak{p}}\sim E^3$, where $E$ is a supersingular elliptic curve.
\end{proposition}

\begin{proof} 
By Proposition \ref{prop: embeddingR}, we have $J_{\mathfrak{p}}\sim E^3$ where $E$ is an elliptic curve. Assume that $E$ is ordinary. Then, $E$ has CM by an order of an imaginary quadratic field $k := \Q(\sqrt{D})$, where $D$ is the fundamental discriminant. 
By \cite[Theorem 2]{Kani}, there are ordinary elliptic curves $E_i$ with CM by $k$ such that $E_i$ is isogenous to~$E$ and $J_\mathfrak{p}\simeq E_1\times E_2\times E_3$. Thus, we have $\text{End}(J_\mathfrak{p})\otimes\Q\simeq\mathcal{M}_{3\times 3}(k)$ and, by \cite[Theorem 1.3.1.1]{ChaiConradOort}, the field $k$ is isomorphic to the center of ${M}_{3\times 3}(k)$ hence isomorphic to a subfield in $K$. Therefore, if $K$ does not contain any imaginary quadratic field, then $E$ is not ordinary and we get a contradiction. This has already been proven in \cite[Proposition 4.6]{BCLLMNO}. 

Now suppose that $K$ contains an imaginary quadratic field. Then, the imaginary quadratic field of $K$ is isomorphic to $k$. We identify $k$ and the imaginary quadratic subfield in $K$. Write $\End(E_i)=\langle 1, f_i\omega\rangle$, where $\omega=\frac{D+\sqrt{D}}{2}$ and $f_i$ is the conductor of $\End(E_i)$ in the maximal order of $k$. Let $f=\operatorname{lcm}(f_1,f_2,f_3)$. Since $k$ is embedded diagonally in $\mathcal{M}_{3\times 3}(k)$, if $p \nmid 2fD$, this contradicts the primitivity of $\Phi$, by \cite[Proposition 5.8]{KLLNOS}. If $p\mid fD$ and $p\neq2$,  then $p\mid f_iD$ for some $i$. Write the Frobenius element $\pi_{E_i}=a+bf_i\omega$. Then, $N(\pi_{E_i})=abf_iD+a^2+b^2f_i^2\frac{D(D-1)}{4}$ is a power of $p$, so $p\mid a$. Hence, $p\mid \Tr(\pi_{E_i})=2a+bf_iD$, which implies that $E_i$ is not ordinary. This is a contradiction, so $E_i$ is supersingular.

\end{proof}

\begin{corollary}\label{cor: pot_good_primes} Let $X$ be a curve of genus 3 with CM by the maximal order of a sextic CM field $K$. Let $M$ be a number field such that the Jacobian $J$ of $X$ has good reduction at every prime of $\O_M$. Let $\frakp_M \subset \O_M$ and let $p:= \frakp_M \cap \Z$ be an odd prime. 
If
$$p \O_K = \frakp_1\overline{\frakp_1}\frakp_2\overline{\frakp_2} \quad \text{or} \quad \frakp_1\overline{\frakp_1}\frakp_2\frakp_3 \quad \text{or} \quad \frakp_1\overline{\frakp_1}\frakp_2\overline{\frakp_2}\frakp_3\quad \text{or} \quad \frakp_1\overline{\frakp_1}\frakp_2\overline{\frakp_2}\frakp_3\overline{\frakp_3}$$
then the curve has potential good reduction at $\frakp_M$.  Moreover, if $K/\Q$ is a Galois extension, then $p \O_K = \frakp\overline{\frakp}$ and
$\frakp^3\overline{\frakp}^3$ also implies that the curve has potential good reduction at $\frakp_M$. 
\end{corollary}

\begin{proof}
Zaytsev \cite{Zaytsev} proves that the $p$-rank of the Jacobian $J \mod \frakp_M$ is different than $0$ in these cases. So by  Proposition~\ref{prop:ss}, we conclude that $X$ does not have bad reduction in these cases. The statement about the Galois case follows from \cite[Proposition 4.1]{KLLRSS}. 
\end{proof}

We now state our first embedding problem for genus 3 CM curves which is a generalization of the embedding problem for genus 2 CM curves given in \cite{GorenLauter07}. 

\begin{problem}\label{Prob:REP} \textit{(The Reduction Embedding Problem)} Given an order $\O$ in a sextic CM field~$K$ and a prime number $p$, the Reduction Embedding Problem (REP) is to find a supersingular elliptic curve $E$ and a principally polarized abelian surface $(A,\lambda_A)$ both defined over $\overline{\mathbb{F}}_{p}$, and a ring embedding
\begin{equation*}
\iota_0 \colon \O \hookrightarrow \End(E\times A)
\end{equation*}
such that the Rosati involution coming from the product polarization $1\times \lambda_A$ induces complex conjugation on $\OO$. We call the tuple $(E\times A, 1\times \lambda_A, \iota_0)$ a solution to the REP for $(\mathcal{O},p)$.
\end{problem}

The following corollary is a direct consequence of Proposition \ref{prop: embeddingR}. 

\begin{cor}\label{badred-REP} 
With the notation and the assumptions of Proposition~\ref{prop: embeddingR}, a genus $3$ curve with CM by $\mathcal{O}\subset K$ and a prime $\mathfrak{p}\mid p$, with $p \neq 2$, of bad reduction for $X$ produce a solution to the REP. 
\end{cor}

We will discuss the converse of the corollary in Section \ref{sec:REPimpliesBadRed}.

The ideas for the proof of the following proposition are given in \cite[Section 3]{KLLNOS}, but we reproduce it here for completeness, using our terminology. 

\begin{proposition} \label{prop: embeddingI}
Let $X$ be a curve of genus $3$ defined over a number field $L$ with CM by an order $\OO$ in a sextic CM field~$K$. Suppose that the Jacobian $J$ of $X$ is absolutely simple. Let $M$ be a  finite extension of $L$ such that $J$ has good reduction everywhere and $X$ has a stable model over $\O_M$. 
If $X$ has bad reduction modulo a prime ideal $\mathfrak{p}\subset \O_M$, with $\mathfrak{p} \nmid 2$, then there exists a supersingular elliptic curve $E$ defined over $\overline{\mathbb{F}}_{p}$ with endomorphism ring $\mathcal{R}$, 
a polarization $\lambda=\begin{pmatrix}\alpha & \beta\\ \beta^\vee & \gamma\end{pmatrix}$ on $E^2$ where $\alpha, \gamma \in \Z_{>0}$, $\beta \in \calR$, $n:= \alpha\gamma - \beta\beta^\vee \in \Z_{>0}$ 
and a ring embedding
\begin{equation*}
\O \hookrightarrow \mathcal{M}_{3\times 3}(\mathcal{R}/n)
\end{equation*}
such that the Rosati involution coming from the polarization $1\times \lambda$ induces complex conjugation on $\OO$ and the entries of the first row are in $\mathcal{R}$.
\end{proposition}

\begin{proof}
Suppose that $X$ has bad reduction modulo $\mathfrak{p} \subset \O_M$. Then, by Proposition \ref{prop: embeddingR} and Corollary \ref{badred-REP}, $(X, \mathfrak{p})$ has a solution to the REP. It is proven in Propositions \ref{prop: embeddingR} and \ref{prop:ss} that $A$ is isogenous to $E^2$ where $E$ is a supersingular elliptic curve over $\overline{\F}_p$. Let $s\colon E^2\rightarrow A$ be a fixed isogeny. Then by \cite[Lemma 4.3]{KLLNOS}, the polarization $\lambda$ induced by $\lambda_A$ on $E^2$ is as given in the statement, and for $n = \alpha\gamma - \beta\beta^\vee \in \Z_{>0}$, there is an isogeny $\tilde{s} \colon A \rightarrow E^2$ be an isogeny such that $\tilde{s}s = [n]$. Then, by sending $\varphi \in \End(E \times A)$ to $\frac{1}{n}(n\times \tilde{s}) \varphi (1 \times s) \in \End(E^3)\otimes \Q = \mathcal{M}_{3\times 3}(\mathcal{R}/n)$, we embed $\End(E \times A)$ into $ \mathcal{M}_{3\times 3}(\mathcal{R}/n)$. Hence, we obtain an embedding 
$\O \hookrightarrow \mathcal{M}_{3\times 3}(\mathcal{R}/n)$. 

It is straightforward to check the Rosati involution condition. The last condition follows from  \cite[Lemma 5-(1)]{KLS}. 
\end{proof}

\begin{remark}\label{E3} Because of Remark \ref{rmk:2cases}, if $p$ does not divide $n$ then $A\simeq E^2$.
\end{remark}

In Section \ref{sec: embedding_img} we will see how to fix a particular isogeny $s\colon E^2\rightarrow A$ starting by a solution to the REP.

Proposition~\ref{prop: embeddingI} is the motivation for introducing the following problem, whose resolution will be the main focus of this paper.

\begin{problem}\label{Prob:IEP}\textit{(The Isogenous Embedding Problem)} Given an order $\O$ in a sextic CM field~$K$ and a prime number $p$, the Isogenous Embedding Problem (IEP) is to find a supersingular elliptic curve $E$ over $\overline{\mathbb{F}}_p$ with $\End(E)\cong \mathcal{R}$, 
a polarization $\lambda=\begin{pmatrix}\alpha & \beta\\ \beta^\vee & \gamma\end{pmatrix}$ on $E^2$ where $\alpha, \gamma \in \Z_{>0}$, $\beta \in \calR$, $n:= \alpha\gamma - \beta\beta^\vee \in \Z_{>0}$, and a ring embedding 
\begin{equation*}
\iota : \O \rightarrow \mathcal{M}_{3\times 3}(\mathcal{R}/n)
\end{equation*}
such that the Rosati involution coming from the polarization $1\times \lambda$ induces complex conjugation on $\OO$  and the entries of the first row are in $\mathcal{R}$. We call the tuple $(E^3,1\times \lambda, \iota : \O \rightarrow \mathcal{M}_{3\times 3}(\mathcal{R}/n))$
a solution to the IEP for $(\mathcal{O},p)$.
\end{problem}

\begin{cor}\label{badred-IEP} With the notation and the assumptions of Proposition~\ref{prop: embeddingI}, a genus $3$ curve with CM by $\mathcal{O}\subset K$ and a prime $\mathfrak{p}\mid p$, with $p \neq 2$, of bad reduction produce a solution to the IEP. 
\end{cor}

We saw that \textit{bad reduction} implies the existence of a solution to the REP and that implies the existence of a solution to the IEP. In order to compute primes of bad reduction for a genus~3 curve with CM by a sextic order $\O$, we compute solutions to the IEP. This will provide a bound for the primes of  bad reduction.

\subsection{From a solution to REP to bad reduction}\label{sec:REPimpliesBadRed}

In genus $2$, a solution to the embedding problem produces a curve with bad reduction~\cite[Theorem 4.2.1]{GorenLauter07}. Concretely, suppose that there exists an embedding $\OO \hookrightarrow \End(E^2)\otimes\Q$, where $\OO$ is the ring of integers of a non-biquadratic quartic CM field $K$ and $E$ is defined over $\overline{\F}_p$, such that the Rosati involution coming from the product polarization induces complex conjugation on $\OO_K$. 
Recall that a CM type of a CM field $K$ of degree $4$ is primitive if and only if $K$ does not contain any imaginary quadratic field. Then, there exists a curve $X$ of genus 2 over a number field $L$ with CM by $\OO$, whose endomorphisms and the stable model are defined over~$\OO_L$, and a prime $\frakp \subset \OO_L$ such that $X$ has bad reduction modulo $\frakp$. In genus 3, bad reduction produces a solution to the REP by Proposition \ref{prop: embeddingR}. However, this strategy cannot be fully used in genus 3 to prove the converse.

\begin{proposition} Let $K$ be a sextic CM field not containing any imaginary quadratic field and $\mathcal{O}\subseteq K$ be an order.
 Let $(E\times A, 1\times \lambda_A, \iota_0: \O \rightarrow \End(E\times A))$ be a solution to the REP, then there exists a genus 3 curve with CM by $\mathcal{O}$ over a number field having bad reduction at a prime $\mathfrak{p}$ dividing $p$. 
\end{proposition}

\begin{proof}
From a solution to the REP there exists (at least) a lift to characteristic zero by \cite[Lemma 4.4.1]{GorenLauter07}. By the Lefschetz principle, this lift can be realized over $\mathbb{C}$ and, by the theory of complex multiplication, it is actually defined over a number field. This produces a principally polarized absolutely simple abelian threefold and hence a Jacobian. The corresponding curve has bad reduction at $p$ and it produces the REP solution.
\end{proof}

When $K$ contains an imaginary quadratic field, it admits both primitive and non-primitive CM-types so we cannot conclude that the CM type of the lift is primitive. 
Even more, in the case that the CM type is non-primitive we cannot deduce that the polarization of the lift is indecomposable hence we cannot conclude that the lift is a Jacobian.

Another issue that deserves more study is the fact that the lift may not be unique, this is related to the generalization of Gross-Zagier formula on singular moduli measuring for which primes different CM curves have isomorphic reduction, see \cite{GrossZagier, LV15a}.

%% file: Sec3-Embedding.tex
Let us consider $(E\times A, 1\times \lambda_A, \iota_0:\,\mathcal{O} \hookrightarrow \text{End}(E\times A))$ a solution to the REP.
In this section, we effectively construct a particular solution to the IEP, starting from the solution to the REP, by fixing a particular isogeny from $E^2$ to $A$ proposed in \cite[Sec. 3]{KLLNOS}. We fix elements $\mu,\eta\in\O$ such that $K=\Q(\eta)$, $K_+=\Q(\mu)$ and $\mu=-\eta^2$.

Recall that $E$ is supersingular and $A\sim E^2$. We denote by $\mathcal{R}\isom \End(E)$ a maximal order in the quaternion algebra $\mathcal{B}_{p,\infty}$. We write 
\begin{equation}\label{eq:xyzw}
  \iota_0(\mu)=:\xyzw,
\end{equation}
where $x\in \mathcal{R}$, $y\in\text{Hom}(A,E)$, $z\in\text{Hom}(E,A)$ and $w\in \text{End}(A)$. In Equation~\eqref{eq:xyzw} the sizes of the boxes reflect the dimensions of the (co)domains of the homomorphisms. 

The Rosati involution is defined by $f \mapsto f^* := \lambda_A^{-1} f^\vee \lambda_A \in \End(E\times A)\otimes \Q$ for $f\in \End(E\times A)$, which induces complex conjugation on $\OO$.
Since $\mu\in K_+$, i.e., $\overline{\mu} = \mu$,  we have 
$$
\iota_0(\mu) = (1 \times \lambda_A)^{-1} \circ \iota_0(\mu)^\vee \circ (1 \times \lambda_A)
$$
and this equality gives
\begin{equation}\label{eq:duals}
x = x^\vee,\, \lambda_A z = y^\vee  \text{ and }\lambda_A w = w^\vee\lambda_A.
\end{equation}

\begin{lemma}\label{isogeny}
The homomorphism
\begin{align*}
s_\mu = \homEEA{z}{\!\! wz \!\!} :E\times E &\longrightarrow A\\
(P,Q)&\longmapsto z(P) + wz(Q)
\end{align*} 
is an isogeny. 
\end{lemma}
\begin{proof}
This is analogous to \cite[Lemma 3.1]{KLLNOS}. \end{proof}

By using the isogeny $s_\mu:\,E^2\rightarrow A$ constructed in Lemma \ref{isogeny}, we consider the isogeny
\begin{align*}
	F = 1 \times s_\mu = \homEEEtoEA{1}{0}{0}{0}{z}{\!\! wz \!\!}
    : E^3 &\longrightarrow E\times A\\
	(P,Q,R) &\longmapsto (P, s_\mu(Q,R)),
\end{align*}
 which induces the following embedding
\begin{align*} 
\iota_1 : \text{End}(E\times A)&\longrightarrow \text{End}(E^3)\otimes \mathbb{Q}\cong\calM_{3\times 3}(\mathcal{B})\\
\varphi&\longmapsto F^{-1} \varphi F,
\end{align*}
where $F^{-1}=\frac{1}{n}\tilde{F}$ and $\tilde{F}$ is the contragradient isogeny of $F$ such that $\tilde{F}F=[n]$. 
Hence, the tuple $(E^3, 1\times \lambda, \iota_1\circ \iota_0:\,\O\hookrightarrow \mathcal{M}_{3 \times 3}(\mathcal{R}/n))$, with $\lambda=s_\mu^{\vee}\circ \lambda_A\circ s_\mu$, is a solution to the IEP (as in the proof of Proposition~\ref{prop: embeddingI}).

More concretely,
\begin{equation}\label{eq:polarization}
\lambda := F^\vee 
\polmatrix F=\begin{pmatrix}1 & 0 & 0 \\ 0 & z^\vee\lambda_A z & z^\vee \lambda_A wz \\ 0 & z^\vee w^\vee \lambda_A z & z^\vee w^\vee \lambda_A wz \end{pmatrix}
=  \begin{pmatrix}1 & 0 & 0 \\ 0 & \alpha & \beta \\ 0 & \beta^\vee & \gamma \end{pmatrix}
\end{equation}
where \label{a_gamma_n}
$\alpha,\gamma\in\Z_{>0}$ and $\beta\in \calR$ such that $n := \alpha\gamma-\beta\beta^\vee\in\Z_{>0}$ and $[n]\ker(F) = 0$. 
Moreover,
\begin{equation}\label{eq:iotamu}
\iota(\mu)=F^{-1}\xyzw F=\begin{pmatrix}
 x & yz & ywz \\ 1 & 0 & c/n \\ 0 & 1 & d/n
 \end{pmatrix}=\begin{pmatrix}
 x & a & b \\ 1 & 0 & c/n \\ 0 & 1 & d/n
 \end{pmatrix} =: \Lambda_1,
\end{equation}
for some $x,a,b,c,d\in \calR$ and such that $w^2z=z\frac{c}{n}+wz\frac{d}{n}$. From \eqref{eq:duals}, \eqref{eq:polarization}, \eqref{eq:iotamu}, we deduce that $\alpha=a$ and $\beta=b$.

%% file: Sec3-1-EmbeddingOfTheCubic.tex

In this subsection, we look closer at the embedding of $\mathcal{O}_+$ by the solution to the IEP $\iota \colon \mathcal{O}\hookrightarrow\mathcal{M}_{3 \times 3}(\mathcal{R}/n)$ that we have previously constructed starting from a solution $\iota_0 \colon \O\hookrightarrow\End(E\times A)$ to the REP. The ideas in this section are again similar to the ones in \cite[Sec. 3]{KLLNOS} and \cite[Sec. 4]{KLS}.
\begin{prop}\label{prop:lambda1}
The coefficients of the matrix $\Lambda_1$ in \eqref{eq:iotamu} lie in $\frac{1}{n}\Z$. 
\end{prop}

\begin{proof}
Since the Rosati involution induces complex conjugation on $\OO$, for every $\theta\in K = \OO\otimes \Q$, we have 
\begin{equation}\label{eq: involution}
\iota(\overline{\theta}) = \lambda^{-1}\iota(\theta)^\vee\lambda.
\end{equation}
Taking $\theta=\mu\ (= \overline{\mu})$, we get $\lambda \iota(\mu)=\iota(\mu)^{\vee}\lambda$, that is,
$$
\left(\begin{array}{ccc}x & a & b\\ \alpha & \beta & \alpha c/n+\beta d/n\\ \beta^{\vee}& \gamma & \beta^{\vee}c/n+\gamma d/n\end{array}\right)=
\left(\begin{array}{ccc} x^{\vee} & \alpha & \beta \\ a^{\vee} &\beta^{\vee}&\gamma\\b^{\vee} & (c^{\vee}/n)\alpha+(d^{\vee}/n)\beta^{\vee}&(c^{\vee}/n)\beta+(d^{\vee}/n)\gamma\end{array}\right).
$$
Hence, we obtain 
\begin{align}
x^\vee & = x \in \Z, a = \alpha = a^\vee \in \Z_{>0}, b = \beta = \beta^\vee \in \Z,\label{relationsFromKLLNOS}\\
n &=\alpha \gamma - \beta^2 = a\gamma - b^2, \label{eq1:n} \\
\gamma &= \frac{ac + bd}{n}. \label{eq1:gamma}
\end{align}
We also derive $ a c^\vee + b d^\vee = n\gamma = a c + b d$ and $bc + \gamma d  = c^\vee b + \gamma d^\vee$. Thus, we obtain $a(c - c^\vee) = b (d^\vee - d)$ and $b(c - c^\vee) = \gamma (d^\vee - d)$. It follows that $b^2 (c - c^\vee) (d^\vee - d) = a\gamma (c - c^\vee) (d^\vee - d)$. Hence, we get $n (c - c^\vee) (d^\vee - d)  = 0.$ Since $n\neq 0$, we get $c = c^\vee$ and $d = d^\vee$, which leads to the conclusion that $x,a,b \in \Z $ and $c/n,d/n \in \frac{1}{n}\Z$. 
\end{proof}

Let $t^6 + At^4 + Bt^2 + C \in \Z[t]$ be the minimal polynomial of $\eta$. Recall that $\eta^2 = -\mu$. The minimal polynomial of $\mu$ is $t^3 - At^2 + Bt - C$ and $A,B,C\geq0$, since $\mu$ is totally positive. By computing the characteristic polynomial of $\iota(\mu)$, which is a matrix with rational coefficients, we obtain
\begin{align}
\Tr_{K_+/\Q}(\mu) &= A = \frac{d}{n} + x, \label{eq:A}\\
 B &= \frac{dx}{n} -\frac{c}{n} - a, \label{eq:B}\\
 \N_{K_+/\Q}(\mu) &= C = b - \frac{cx}{n} - \frac{ad}{n}. \label{eq:C}
\end{align}

\begin{cor}\label{Cor:ABC}
We have $n \mid d$ and $n \mid c$. In particular, we have $x,a,b,\frac{c}{n},\frac{d}{n} \in \Z$.
\end{cor}

\begin{proof}
By \eqref{eq:A}, we have $\frac{d}{n} = A - x \in \Z$ and by \eqref{eq:B}, we get $\frac{c}{n}  = \frac{dx}{n} - a - B \in \Z$.
\end{proof}

In order to find all the solutions to the embedding problems for a given CM field $K$, we first determine all possible embeddings of $\mu$. In particular, we determine all possible values for $[x,a,b,c,d,n]$. Then, we check if these embeddings of $\Z[\mu]$, also {extend} to  embeddings of $\O_{+}$ into $\mathcal{M}_{3 \times 3}(\mathcal{R}/n)$.

%% file: Sec3-2-EmbeddingOfO_eta.tex
We now study the embedding of $\eta$. Let
\begin{eqnarray}\label{iota_of_eta}
\Lambda_2 := \iota(\eta) = \begin{pmatrix} x_1 & x_2 & x_3 \\ x_4 & x_5 & x_6 \\ x_7 & x_8 & x_9 \end{pmatrix}
\end{eqnarray}
where $x_1, x_2, x_3 \in\calR$ and $x_4, \ldots, x_9\in\frac{1}{n}\calR$. We denote  the reduced trace  and  norm on $\mathcal B = \calR \otimes \Q$ by  $\Tr$  and  $\N$.

\begin{prop}\label{prop: lambda2}
We have
$$
\Lambda_2 = \begin{pmatrix} 
x_1 & x_2 &  x_3\\ 
\frac{\gamma}{n}x_2 - \frac{b}{n}x_3 & x_1 + \left( k_1\frac{c}{n} - k_2\frac{d}{n} \right)x_2 + k_2x_3 & \left( k_2 \frac{c}{n} - \frac{b^2}{n} \right)x_2 + \left( \frac{ab}{n} + k_1\frac{c}{n} \right)x_3 \\ 
-\frac{b}{n}x_2 + \frac{a}{n}x_3 & k_2 x_2 + k_1x_3  & x_1 + \left(\frac{ab}{n} + k_1\frac{c}{n}  \right)x_2 - \left(\frac{a^2}{n} - k_1\frac{d}{n} - k_2 \right)x_3 \end{pmatrix},
$$
where $x_1, x_2, x_3 \in \calR$ with $\Tr(x_1) = \Tr(x_2) = \Tr(x_3) = 0$, $k_1 = \frac{ad}{n^2}-\frac{ax}{n}-\frac{b}{n}$, $k_2 = \frac{ac}{n^2} +\frac{bx}{n}$.
\end{prop}

\begin{proof}
Throughout this proof we make heavy use of equations \eqref{eq1:n} and \eqref{eq1:gamma}. We start by noting that, since $\Tr_{K/\Q}(\eta) = 0$, we have
\begin{equation} \label{eq: Tr_eta}
\Tr(x_1) + \Tr(x_5) + \Tr(x_9) = 0.
\end{equation}

Secondly, using the involution equality (\ref{eq: involution}) and the fact that $\overline{\eta} = -\eta$, we obtain 
$$
- \begin{pmatrix}1 & 0 & 0 \\ 0 & a & b \\ 0 & b & \gamma \end{pmatrix} 
\begin{pmatrix} x_1 & x_2 & x_3 \\ x_4 & x_5 & x_6 \\ x_7 & x_8 & x_9 \end{pmatrix} 
= 
\begin{pmatrix} x_1^\vee & x_4^\vee & x_7^\vee \\ x_2^\vee & x_5^\vee & x_8^\vee \\ x_3^\vee & x_6^\vee & x_9^\vee \end{pmatrix} 
\begin{pmatrix}1 & 0 & 0 \\ 0 & a & b \\ 0 & b & \gamma \end{pmatrix}.
$$
Thus, $x_1^\vee  = -x_1$, which implies $\Tr(x_1) = 0$. Hence, by \eqref{eq: Tr_eta}, $\Tr(x_5) = - \Tr(x_9)$. From the matrix equality above, we also get
\begin{align}
x_2^\vee & = -ax_4 - bx_7, \label{x_2}\\
x_3^\vee & = -bx_4 - \gamma x_7, \label{x_3}\\
a\Tr(x_5) &= - b\Tr(x_8), \label{x_5x_8} \\
b\Tr(x_6) &= - \gamma \Tr(x_9), \nonumber \\
bx_5^\vee + \gamma x_8^\vee & = - ax_6 - bx_9.  \nonumber
\end{align}

Since $\eta$ and $\mu$  commute, we have $\Lambda_1\Lambda_2 = \Lambda_2 \Lambda_1$,  which leads to the following relations:
\begin{align} 
ax_4 + bx_7 &= x_2, \label{eqq:x2}\\ 
xx_2 +ax_5 + bx_8 &= ax_1 + x_3, \label{eqq:x3} \\
xx_3 + ax_6 + bx_9 &= bx_1 + (c/n)x_2 + (d/n)x_3, \nonumber \\
x_1 + (c/n)x_7 &= xx_4 + x_5,  \label{eqq:x1} \\ 
x_2 + (c/n)x_8 &=ax_4 + x_6, \nonumber \\ 
x_3 + (c/n)x_9 &=bx_4 + (c/n)x_5 + (d/n)x_6, \nonumber \\ 
x_4 + (d/n)x_7 &= xx_7 + x_8, \label{eqq:x8} \\
x_5 + (d/n)x_8 &=ax_7 + x_9,  \label{eqq:x9} \\
 x_6 + (d/n)x_9 &=bx_7 + (c/n)x_8 + (d/n)x_9. \label{eqq:x6}
\end{align}
Adding equations \eqref{x_2} and \eqref{eqq:x2}, we obtain $\Tr(x_2) = 0$. Furthermore, applying the trace operator to \eqref{eqq:x3} and using $\eqref{x_5x_8}$ together with the fact that $\Tr(x_1)=0$, we get
$\Tr(x_3)=0$. Multiplying \eqref{x_3} by $\frac{a}{n}$ and \eqref{x_2} by $\frac{-b}{n}$, adding the two equations together and taking the dual, we obtain $ - x_7^{\vee} = - \frac{b}{n} x_2 + \frac{a}{n}x_3$. Hence, $\Tr(x_7)=0$ and 
$ x_7 = - \frac{b}{n}x_2+ \frac{a}{n}x_3. $
Substituting this in \eqref{eqq:x2}, we obtain
$ x_4 = \frac{\gamma}{n}x_2- \frac{b}{n}x_3. $
Then, it follows from \eqref{eqq:x1} and \eqref{eqq:x8} that
$x_5 = x_1 + \left( k_1\frac{c}{n} - k_2\frac{d}{n} \right)x_2 + k_2x_3,$ and 
$x_8 = k_2 x_2 + k_1x_3,$
where $k_1 = \frac{ad}{n^2}-\frac{ax}{n}-\frac{b}{n}$, $k_2 = \frac{ac}{n^2} +\frac{bx}{n}$. Substituting these in  \eqref{eqq:x9} leads to
$x_9 =  x_1 + \left(\frac{ab}{n} + k_1\frac{c}{n}  \right)x_2 - \left(\frac{a^2}{n} - k_1\frac{d}{n} - k_2 \right)x_3. $
Finally, \eqref{eqq:x6} implies
$x_6 = \left( k_2 \frac{c}{n} - \frac{b^2}{n} \right)x_2 + \left( \frac{ab}{n} + k_1\frac{c}{n} \right)x_3.$ 
We remark that the remaining unnumbered equations are satisfied and do not yield new relations.
\end{proof}

We now relate the coefficients of $\Lambda_1$ and $\Lambda_2$.

\begin{prop} \label{prop: relations} We have the following system of equations: 
\begin{align*}
x = &  \N(x_1) + \frac{\gamma}{n}\N(x_2) + \frac{a}{n} \N(x_3) + \frac{b}{n}\Tr(x_2x_3), \\
a =&  \left( k_1\frac{c}{n} - k_2\frac{d}{n} \right)\N(x_2) + k_1 \N(x_3)  - k_2\Tr(x_2x_3) -\Tr(x_1x_2),\\ 
b =& \left( k_2\frac{c}{n} - \frac{b^2}{n}\right)\N(x_2) - \left(\frac{a^2}{n} - k_1\frac{d}{n} - k_2 \right)\N(x_3)    - \left(\frac{ab}{n} + k_1\frac{c}{n} \right) \Tr(x_2x_3) - \Tr(x_1x_3). 
\end{align*}
\end{prop}

\begin{proof}
The relation $\eta^2 = -\mu$ implies that $\Lambda_2^2 = -\Lambda_1$. Thus, we obtain a system of nine equalities. Using the equality $1-ak_2 = bk_1$, one can show that these nine equalities reduce to the system of three equations above. 
\end{proof}

\begin{defn}\label{def:sol} 
Let $K$ be a sextic CM field and $\O$ be an order in $K$. Let $\eta, \mu\in \O$ such that $\eta$ is totally imaginary and $\eta^2 = -\mu$. Let $p$ a prime number and $E$ a supersingular elliptic curve defined over $\mathbb{F}_{p^2}$ and denote by $\calR$ a maximal order in a quaternion algebra such that $\End(E)\isom \calR$. Let $\lambda = \begin{pmatrix}a & b\\ b & \gamma\end{pmatrix}$ be a polarization on~$E^2$, where $a,b, \gamma \in \Z_{>0}$, and $n:= a\gamma - b^2 \in \Z_{>0}$.  Let $(E^3, 1\times \lambda, \iota: \Z[\eta]\rightarrow \mathcal{M}_{3 \times 3}(\calR/n))$ be a solution to the IEP  satisfying the following:
$\iota(\eta)=\Lambda_1$ and $\iota(\mu)=\Lambda_2$ as in equations \eqref{eq:iotamu} and \eqref{iota_of_eta} with
\begin{itemize}
\item  $x,a,b,c,d,n \in \Z $ with $n\mid c$ and $n\mid d$ satisfy \eqref{relationsFromKLLNOS}--\eqref{eq:C};
\item $x_1, x_2, x_3 \in \calR$ with $\Tr(x_i)=0$ satisfying the equations in Proposition \ref{prop: relations}.
\end{itemize}

We call this solution a \emph{\Int}~solution for $\Z[\eta]$. An IEP solution for $(\O,p)$ is called \emph{\Int}~if its restriction to $\Z[\eta]$ is $\Int$. A \Int~IEP solution for $(\O,p)$ is denoted by a tuple $(E,\eta,[x,d/n,a,c/n,b,\gamma,n,x_1,x_2,x_3])$.
\end{defn}

We check whether a \Int~solution for $\Z[\eta]$ induces a \Int~IEP solution for $(\O,p)$ using the arguments in Section~\ref{sec: lifting arguments}.

The following result can be read along the lines of the proof of Theorem~1.1 in \cite{KLLNOS}. We explicitly state the result and the proof for the sake of completeness.

\begin{prop} \label{prop: xi_non-comm} Let $K=\mathbb{Q}(\eta)$ be a sextic CM-field and  $(E,\eta,[x,d/n,a,c/n,b,\gamma,n,x_1,x_2,x_3])$ a \Int~solution. If $K$ does not contain an imaginary quadratic field, then $x_i, x_j$ do not commute for $i,j \in \{1,2,3\}$. If it does and $\mathbb{Z}[f\sqrt{D}]\subseteq\mathbb{Z}[\eta]$ then $p$ divides $2nfD$ or $x_i, x_j$ do not commute for $i,j \in \{1,2,3\}$. 
\end{prop}

\begin{proof}
Let $\iota$ denote the IEP embedding of $\O$ into $\mathcal{M}_{3\times 3}(\mathcal{R}/n)$. Since $\iota$ is a ring homomorphism, we have $h(\iota(\eta)) = 0$ where $h$ is the minimal polynomial of $\eta$ which has degree 6 over $\Q$. Therefore, if~$x_i, x_j$ commute for any $i,j \in \{1,2,3\}$, then $x_i$ lies in a field $\mathcal{B}_0 \subset \mathcal{R}\otimes \Q$ of degree~$2$ over $\Q$. This implies that $K$ contains an imaginary quadratic field isomorphic to~$\mathcal{B}_0$. So when $K$ does not contain an imaginary quadratic field we get a contradiction. 

If $K$ contains an imaginary quadratic field and  $p$ does not divide $2nfD$, then by the arguments in the proof of Theorem~1.1 on page 12 in \cite{KLLNOS}, we obtain that $\iota(\eta)$ is either a root of a linear or degree 2 polynomial over~$\mathcal{B}_0$. This contradicts with the degree of~$h$.
\end{proof}

%% file: Sec4-SexticWithImaginaryQuadSubfield.tex

In this section, we suppose that the CM field $K$ contains an imaginary quadratic field $k = \Q(\sqrt{D})$, where $D\in \Z_{<0}$ is square-free. We embed the cubic order $\OO_+$ and the quadratic order $\OO\cap \Q(\sqrt{D})$ instead of embedding the sextic order $\OO$. By this way we obtain better bounds for the primes of bad reduction (see Section~\ref{Ss:Bounds}) and a more effective algorithm (see Remark \ref{remark: imaginary}).

Let $f$ be the conductor of the order $\OO \cap k$ in $\OO_k$. Then, $\OO \cap k =  <1,f\omega>$, where $\omega=\frac{1+\sqrt{D}}{2}$ if $D \equiv 1 \mod 4$ and $\omega=\sqrt{D}$ if $D \equiv 2,\, 3 \mod 4$.  Thus, we can write $\mathcal{O}_K = \mathcal{O}_{K_+}[f\omega]$. 
Let $\iota:\O\hookrightarrow \mathcal{M}_{3\times 3}(\mathcal{R}/n)$ be a solution to the IEP coming from a solution the  REP via an isogeny $s_\mu$, as in Lemma \ref{isogeny}. Let
\begin{equation}\label{iota_of_fsqrtD}
\Lambda_3 := \iota(f\sqrt{D}) = \begin{pmatrix} d_1 & d_2 & d_3 \\ d_4 & d_5 & d_6 \\ d_7 & d_8 & d_9 \end{pmatrix} 
\end{equation}
where all $d_i \in \mathcal R/n$ and $d_1, d_2, d_3 \in \calR$. 
\begin{prop} \label{prop: lambda3}
We have
$$
\Lambda_3 = \begin{pmatrix} 
d_1 & d_2 &  d_3\\ 
\frac{\gamma}{n}d_2 - \frac{b}{n}d_3 & d_1 + \left( k_1\frac{c}{n} - k_2\frac{d}{n} \right)d_2 + k_2d_3 & \left( k_2 \frac{c}{n} - \frac{b^2}{n} \right)d_2 + \left( \frac{ab}{n} + k_1\frac{c}{n} \right)d_3 \\ 
-\frac{b}{n}d_2 + \frac{a}{n}d_3 & k_2 d_2 + k_1d_3  & d_1 + \left(\frac{ab}{n} + k_1\frac{c}{n}  \right)d_2 - \left(\frac{a^2}{n} - k_1\frac{d}{n} - k_2 \right)d_3 \end{pmatrix}
$$
where $d_1, d_2, d_3 \in \calR$ with $\Tr(d_1) = \Tr(d_2) = \Tr(d_3) = 0$, $k_1 = \frac{ad}{n^2}-\frac{ax}{n}-\frac{b}{n}$, $k_2 = \frac{ac}{n^2} +\frac{bx}{n}$.
\end{prop}

\begin{proof}
Since $\Tr_{K/\Q}(f\sqrt{D}) = 0$, $\overline{f\sqrt{D}} = -f\sqrt{D}$ and $\Lambda_3\Lambda_1 = \Lambda_1\Lambda_3$, the proof is the same as that of Proposition \ref{prop: lambda2}.
\end{proof}

\begin{proposition}\label{prop: relations_d}
We have
\begin{align*}
-f^2D =& \N(d_1) + \frac{\gamma}{n}\N(d_2) + \frac{b}{n}\Tr(d_2d_3) + \frac{a}{n}N(d_3)\\
\Tr(d_1d_2) =&  \left(k_1\frac{c}{n} - k_2\frac{d}{n} \right)\N(d_2) + k_1 \N(d_3)  - k_2\Tr(d_2d_3), \\
\Tr(d_1d_3) =& \left( k_2\frac{c}{n} - \frac{b^2}{n}\right)\N(d_2) - \left(\frac{a^2}{n} - k_1\frac{d}{n} - k_2 \right)\N(d_3)    - \left(\frac{ab}{n} + k_1\frac{c}{n} \right) \Tr(d_2d_3).
\end{align*}
\end{proposition} 

\begin{proof}
Similarly to Proposition \ref{prop: relations}, we have $\Lambda_3^2 = f^2D\I_3$ which leads us to the system of three equations above. 
\end{proof}

\begin{remark} \label{remark: sol_imag} In the situation that $K$ contains an imaginary quadratic field, and with the notation above a \Int~solution as in Definition \ref{def:sol} is determined by 
$\iota(\mu)=\Lambda_2$ and  $\iota(f\sqrt{D})=\Lambda_3$. We then describe the solution by a tuple $(E,\mu, f\sqrt{D},[x,d/n,a,c/n,b,\gamma,n,d_1,d_2,d_3])$ with $x,a,b,c,d,n \in \Z $ with $n\mid c$ and $n\mid d$ satisfying \eqref{relationsFromKLLNOS}--\eqref{eq:C} and $d_1, d_2, d_3 \in \calR$ with $\Tr(d_i)=0$ satisfying the equations in Proposition~\ref{prop: relations_d} and such that $\iota(f\omega)\in\mathcal{M}_{3\times3}(\mathcal{R}/n)$. By abuse of notation we also call a solution given by such a tuple a \Int~solution. 
\end{remark}

We check whether a \Int~solution for $\Z[\mu, f\sqrt{D}]$ induces a \Int~solution for $(\OO, p)$ using the arguments in Section \ref{sec: lifting arguments}. 

\begin{remark}
As in Proposition~\ref{prop: xi_non-comm}, if $(E,\mu, f\sqrt{D},[x,d/n,a,c/n,b,\gamma,n,d_1,d_2,d_3])$ is a \Int\, solution and $p\nmid 2fD$, then $d_i, d_j$ do not commute for $i,j \in \{1,2,3\}$. 
\end{remark}

The following lemma follows from \cite[Lemma 5.1 and Proposition 5.8]{KLLNOS} and we use the result in Section~\ref{sec: bounds_for_xi} to get better bounds for the primes of bad reduction when $K$ contains an imaginary quadratic field. 

\begin{proposition}\label{prop: primitivity} Let $\O$ be an order in a sextic CM field containing an imaginary field $\Q(\sqrt{D})$, where $D\in \Z_{<0}$ is square-free. Let $X$ be a curve such that the Jacobian is absolutely simple and has CM by $\O$. Let $\frakp \subset \O$ (lying above rational prime $p$) be a prime of bad reduction for $X$. 

Suppose that $(E,\mu, f\sqrt{D},[x,d/n,a,c/n,b,\gamma,n,d_1,d_2,d_3])$ is a \Int~IEP solution corresponding to $(\OO, p)$. If $p \nmid 2fDn$ then $d_2$ and $d_3$ are not both $0$ in $\Lambda_3 = \iota(f\sqrt{D})$.

\qed
\end{proposition} 

\subsubsection{Extra automorphisms}\label{Subsec:extraaut}
In this section, we assume that there exists a root of unity different from $\pm1$ in $\O$, or alternatively that the Jacobian of the curve has extra automorphisms beyond multiplication by $-1$.
If $\mathcal{O}$ is a sextic order containing a primitive root of unity $\zeta_n$ for $n \geq 3$, then, a quick computation shows that $n \in \{3,4,6,7,9,14,18\}$. If $X$ is a curve of genus $3$ and $\End(J) \cong \mathcal{O}$ with $\zeta_3 \in \mathcal{O}$, then $X$ is a Picard curve, see~\cite{KLS}. 
(In the notation of the previous section, in this case $D = -3$.)
Therefore, if the Jacobian of a curve $X$ admits an automorphism of order 3, 6, 9, or 18, then $X$ is a Picard curve.
In addition, if $\mathcal{O}$ contains $\zeta_{14}$, then it also contains $\zeta_7 = \zeta_{14}^2$.
In that case, $\mathcal{O}$ must be the ring of integers of $K = \mathbb{Q}(\zeta_7)$, and the unique curve with absolutely simple Jacobian whose endomorphism ring is isomorphism to $\mathcal{O}$ is the hyperelliptic curve $y^2=x^7-1$, up to isomorphism.
(In the notation of the previous section, in this case $D = -7$.)
Finally, under the condition that its Jacobian is absolutely simple, a curve $X$ whose Jacobian admits an automorphism of order 4 is hyperelliptic \cite{Weng}.
(In the notation of the previous section, in this case $D = -1$.)

\begin{proposition} \label{prop: lambda3_autom} If $\O$ contains the primitive root of unity $\zeta_4=\sqrt{-1}$ or $\zeta_7$ and $\iota$ comes from a solution to the REP coming from bad reduction for $\mathfrak{p}\subset \O$, then we have
$$
 \iota(\zeta)=\begin{pmatrix}
 * & 0 & 0\\ 0 & * & * \\0 & * & *
 \end{pmatrix}\in\operatorname{End}(J_\mathfrak{p}),
 $$
 for $\zeta= \zeta_4$ or $\zeta= \zeta_7$.
\end{proposition}

\begin{proof}
Let $X$ be a genus $3$ curve with CM by $\O$ an order of a sextic CM field $K$ with a root of unity different from $\pm 1$. Let $\mathcal{X}$ be the stable model of $X$
 and $\mathcal{J}$ the stable (N\'eron) model of its Jacobian. Let $\zeta\in\O$ be a root of unity. Then, Torelli's Theorem implies that there exists an automorphism of the curve $\psi$ that induces $\zeta$ or $-\zeta$ in the Jacobian. There in an injection of $\operatorname{End}(\mathcal{J}_0)\hookrightarrow\operatorname{End}(\mathcal{J}_{\mathfrak{p}})$, so we can see again $\zeta$ as an endomorphism of $\mathcal{J}_{\mathfrak{p}}$, and because it respects the polarization, we can see $\psi$ as an isomorphism of $\mathcal{X}_{\mathfrak{p}}$. 
 
If the positive genus components of $\mathcal{X}_{\mathfrak{p}}$ are $X_1$ and $X_2$, with $g(X_i)=i$, so $\mathcal{J}_{\mathfrak{p}}=J(X_1)\times J(X_2)$, then $\psi(X_1)=X_1$, $\psi(X_2)=X_2$ and hence $\zeta(J(X_1))=J(X_1)$ and $\zeta(J(X_2))=J(X_2)$. 
This implies 

$$
 \zeta=\begin{pmatrix}
 * & 0 & 0\\ 0 & * & * \\0 & * & *
 \end{pmatrix}\in\operatorname{End}(J(X_1)\times J(X_2)).
 $$

 Now, if the positive genus components of $\mathcal{X}_{\mathfrak{p}}$ are $E_1,\,E_2$ and $E_3$, then $\psi$ permute these three elliptic curves. If $D$ is $-1$ or $-7$, then the root of unity has order not dividing $3$ which means that $\Aut(C)$ is either fixing $E_1, E_2$ and $E_3$ or fixing one of them and swapping the other two. Therefore, we may assume that
 $$
 \zeta=\begin{pmatrix}
 * & 0 & 0\\ 0 & * & * \\0 & * & *
 \end{pmatrix}\in\operatorname{End}(E_1\times E_2\times E_3),
 $$
up to permutation.
\end{proof}
 
\begin{corollary} \label{cor: lambda3} Suppose that 
$(E,\mu, \sqrt{D'},[x,d/n,a,c/n,b,\gamma,n,d_1,d_2,d_3])$ is a \Int~IEP solution for $(\O, p)$ coming from a bad reduction. If $\O$ contains $i$ (respectively, $\zeta_7$), i.e. $D'=-1$ (respectively, $D'=-7$)  then
we have
$d_2=d_3=0$ and $p \mid 2n$ (respectively, $p \mid 14n$).

\end{corollary}

\begin{proof} Propositions \ref{prop: lambda3} and \ref{prop: lambda3_autom} give
$d_2=d_3=0$. 
The second part follows from ~Proposition~\ref{prop: primitivity}.
\end{proof}

\subsection{From a solution to the IEP to a solution to the REP} \label{sec: IEPtoREP}
Starting from a solution to the IEP, we can try to get a solution to the REP by computing an isogeny $s:\,E^2\rightarrow A$ of degree $n$. With notation as in the previous section and by defining $\iota_0=\frac{1}{n}\begin{pmatrix}1 & 0\\ 0 & s\end{pmatrix}\iota\begin{pmatrix}n & 0\\ 0 & \tilde{s}\end{pmatrix}$, we need to check that the image of $\iota_0$ belongs to $\operatorname{End}(E\times A)$ and not only in $\operatorname{End}(E\times A)/n$. In order to obtain an $s=\begin{pmatrix}z & wz\end{pmatrix}$, we need $\lambda=s^\vee\lambda_As$ to factor through $s$, $a=yz$ and $b=ywz$, see below for a precise example.

We do not know if it is always possible to get a solution to the REP by starting from a solution to the IEP, neither if the REP solution producing this IEP solution is unique up to equivalence (see Section \ref{sec:eq}). However, Corollary \ref{badred-IEP} implies that the existence of a curve with CM by $\mathcal{O}$ and bad reduction at $\mathfrak{p}\mid p$ gives a solution to the IEP. Hence, the absence of a solution to the IEP implies potentially good reduction. 

\subsubsection{A solution from closer}\label{sec:closeex}
Let $K=\Q(\eta)$ given by $\eta^6 + 13\eta^4 + 50\eta^2 + 49=0$. Note that $\Q(i) \subset K$. The unique curve with CM by $\OO_K$ is given by the equation
$$C: v^2=u^7 + 1786u^5 + 44441u^3 + 278179u.$$
This curve has bad reduction at $7$ and at $11$. Let us consider a \Int~solution to the IEP  $[x,a,\gamma,n,b,c/n,d/n]=[ 5, 2, 16, 7, 5, -12, 8]$ and $[d_1,d_2,d_3]=[i,0,0]$ for $p=7$. Thus,
$$
\iota(\mu)=\begin{pmatrix}
5 & 2 & 5\\
1 & 0 & -12\\
0 & 1 & 8
\end{pmatrix}\,\text{and}\,\,
\iota(i)=\begin{pmatrix}
i & 0& 0\\
0 & i & 0 \\
0 & 0 & i\end{pmatrix}
$$
By direct computations of the reduction (special fiber of the stable model) or by \cite[Example 6.8]{Lor20} or by checking that the polarization is decomposable, we know that the special fiber of the stable model of $C$ is the union of three elliptic curves with endomorphism algebra containing $\Q(i)$. There is only one elliptic curve with endomorphism ring containing $\Z[i]$ over $\overline{\F}_{7}$ and it is supersingular. We denote this curve by $E$ and set $\mathcal{R}=\End(E)\subseteq\mathcal{B}_{7,\infty}$. Hence, the Jacobian of the curve $C$ has reduction at 7 such that $J_7\simeq E\times E\times E$, where the polarization on $E^3$ is the product polarization. Therefore, with the notation in Section~\ref{sec: embedding_cubic}, and more precisely, by \eqref{eq:iotamu}, we have that $y=z^\vee$, $w^{\vee}=w,$ $a=z^{\vee}z$, $b=z^{\vee}wz$, and $\gamma=z^{\vee}w^2z$.

For the chosen solution we have $a=2$. Thus, if $z=(z_1,z_2)$ with $z_i\in\End(E)$, we have two cases: $|z_1|=|z_2|=1$ or $z_1z_2=0$. The second case gives a contradiction, since it yields $b=5=2w_{11}$ or $2w_{22}$ and $5$ is odd. Now, since $|z_1|=|z_2|=1$, they define isomorphisms and up to isomorphisms of the second and third copy of $E$, we can assume $z_1=z_2=1$. From the equations for $b$, $c/n$ and $d/n$ we obtain $ w_{11}+w_{22}=8$, $\Tr(w_{12})=-3$ and $|w_{12}|=w_{11}w_{22}-12$. This implies $\{w_{11},w_{22}\}=\{4,4\}$ or $\{3,5\}$ in order to have $|w_{12}|>0$. There is no element of trace $-3$ and norm $3$ in the maximal ideal $\mathcal{R}$ of $\mathcal{B}_{7,\infty}$, so, up to conjugation $w_{11}=w_{22}=4$ and $w_{12}=\frac{-3+j}{2}$. 
Thus the corresponding REP solution is
$$
\iota_0(\mu)=\begin{pmatrix}
5 & 1 & 1\\
1 & 4 & \frac{-3+j}{2}\\
1 & \frac{-3-j}{2} & 4
\end{pmatrix}
\,\text{and}\,\,\iota_0(i)=\begin{pmatrix}
i & 0& 0\\
0 & 0 & i \\
0 & i &
 0\end{pmatrix} .
$$

\begin{remark} If $p\nmid n$, then the reduction $J_{\mathfrak{p}}$ of the Jacobian is isomorphic to $E\times E\times E$ by Remark \ref{E3}. We implemented the strategy developed in the previous example to compute the lifts of a given solution to the REP to all possible solutions to the IEP giving it. We ran this for all solutions to the IEP with $p$ not dividing $n$ in Tables \ref{table: cubicANDimagfields} and \ref{table: planequarticfields} and we always found a unique (up to equivalence, see Section \ref{sec:eq}) lift to a solution to the REP.
\end{remark}

%% file: Sec5-Equivalences.tex

Let $X/M$ be a genus $3$ curve with CM by a sextic CM order $\mathcal{O}$ and a primitive CM type and $\mathfrak{p}\mid p$ a prime of $M$ of bad reduction for $X$. In Section \ref{Sec: Embedding Problem}  we discussed how this phenomenon gives rise to a solution to the REP problem, see Corollary~\ref{badred-REP}, and to the IEP problem, see Corollary~\ref{badred-IEP}. In this section, we discuss the different choices in the process to obtain the corresponding solutions to these problems and we define as equivalent ones the different solutions that could have been obtained from $(X,\mathfrak{p})$ by making them differently.

When we define a REP solution from $(X,\mathfrak{p})$, see Proposition \ref{prop: embeddingR}, we are fixing isomorphisms $\End(J)\simeq\mathcal{O}$ and $\bar{J}\simeq E\times A$ (as ppav). The first isomorphism makes a choice for $\eta$ among its conjugates in $\mathcal{O}$, while the second one picks the first elliptic curve and an order for the other two if $\lambda_A$ is decomposable. We would like to see all solutions coming from these different choices as equivalent. 

\begin{defn}\label{def:eqREP}
Two solutions $(E\times A, 1\times \lambda_A, \iota_0: \mathcal{O} \hookrightarrow \End(E\times A))$ and  $(E'\times A', 1\times \lambda_{A'}, \iota'_0: \mathcal{O} \hookrightarrow \End(E'\times A'))$ to the REP for $(\mathcal{O},p)$ are \emph{equivalent} if there exist
\begin{enumerate}
\item an automorphism $\sigma\in\Aut_\Z(\O)$,
\item an isomorphism $\phi_0:\,E'\times A'\rightarrow E\times A$ as principally polarized abelian varieties, i.e. we have  $1\times\lambda_{A'}=\phi_0^\vee(1\times \lambda_A)\phi_0$, and
\end{enumerate}
such that $\iota'_0(\theta)=\phi_0^{-1}\iota_0({\,^\sigma\theta})\phi_0$ for all $\theta\in \O$. 
\end{defn}

The group $\Aut_\Z(\O)$ is the subgroup of $\Gal(K^{\Gal}/K)$ leaving $\O$ invariant. Eventually, it may be trivial. If the polarization $\lambda_A$ is indecomposable then $E\simeq E'$ and $A\simeq A'$ and $\phi_0$ is the product of these isomorphisms in Definition \ref{def:eqREP}-(2).  
If $\lambda_A$ is decomposable, we can write $A=E_2\times E_3$ with $\lambda_A$ the product polarization and $\phi_0$ is just a permutation of the elliptic curves (up to isomorphism of the elliptic curves factors). The polarization $\lambda_A$ being indecomposable or not determines the reduction type of the curve, namely, falling into cases $(2)$ or $(3)$ of Proposition \ref{reductionResult}. 

\begin{defn}\label{def:eqIEP}
Two solutions $(E, 1\times \lambda , \iota: \mathcal{O} \hookrightarrow \mathcal{M}_{3\times3}(\mathcal{R}/n))$ with $\mathcal{R}\simeq\End(E)$ and  $(E', 1 \times \lambda', \iota: \mathcal{O}\hookrightarrow    \mathcal{M}_{3\times3}(\mathcal{R}'/n))$ with $\mathcal{R}'\simeq\End(E')$ to the IEP for $(\mathcal{O},p)$ are \emph{equivalent} if there exist
\begin{enumerate}
\item an automorphism $\sigma\in\Aut_\Z(\O)$,
\item a homomorphism $\phi:\Hom^0(E'^3,E^3)$ as principally polarized abelian varieties, i.e. we have $1\times \lambda' =\phi^\vee(1\times \lambda)\phi $,
\item an automorphism $\varphi\in\Aut(\mathcal{R}')$ inducing an automorphism $\varphi\in\Aut(\mathcal{M}_{3\times3}(\mathcal{R}'))$, 
\end{enumerate}
such that $\iota'(\theta)=\varphi(\phi^{-1}\iota({\,^\sigma\theta})\phi)$ for all $\theta\in \O$. 
\end{defn}

We show that conditions (1) and (2) in Proposition \ref{def:eqIEP} are naturally induced by conditions (1) and (2) in Proposition \ref{def:eqREP}.

\begin{proposition}\label{prop:IEP}
Let $(E\times A, 1\times \lambda_A, \iota_0: \mathcal{O} \hookrightarrow \End(E\times A))$ and  $(E'\times A', 1\times \lambda_{A'}, \iota'_0: \mathcal{O} \hookrightarrow \End(E'\times A'))$ be two solutions to the REP which are equivalent in the sense of Definition ~\ref{def:eqREP}. Then the two induced {(by the same element $\eta\in\mathcal{O}$ as in Section \ref{sec: embedding_img})} \Int~solutions  $(E^3,1\times \lambda,\iota:\mathcal{O}\hookrightarrow \mathcal{M}_{3\times 3}(\mathcal{R}/n))$ and $(E'^3,1\times \lambda', \iota':\mathcal{O} \hookrightarrow \mathcal{M}_{3\times 3}(\mathcal{R}'/n))$ are equivalent. 
\end{proposition}
\begin{proof}
If the two solutions to the REP verify condition (1) in Definition \ref{def:eqREP}, then it is straightforward to check that the induced solutions to the IEP verify condition (1) in Definition \ref{def:eqIEP}. Assume now that the solutions to the REP verify condition (2) in Definition \ref{def:eqREP} and let  $\phi_0:\,E'\times A'\rightarrow E\times A$ the isomorphism of principally polarized abelian varieties such that $\iota'_0(\cdot)=\phi_0^{-1}\iota_0(\cdot)\phi_0$. 
\begin{itemize}
\item [(a)] 
 Assume first that $A$ and $A'$ are indecomposable. 
 Denote $\phi_0: =\begin{pmatrix} \phi_{1} & 0 \\ 0 & \phi_{2} \\ \end{pmatrix}:\,E'_0 \times A' \rightarrow E\times A$. 
 Recall that $s(P,Q)=z(P)+wz(Q)$. Then $s'=z'(P)+w'z'(Q)=\phi_2^{-1}z\phi_1(P)+\phi^{-1}_2 wz\phi_1(Q)=\phi_2^{-1}s(\phi_1(P),\phi_1(Q))$.
 We conclude that $s'=\phi_2^{-1}s(\phi_1 \times \phi_1)$.
To check the equation verified by the polarizations, recall that 
$1\times\lambda_{A'}=\phi_0^\vee(1\times \lambda_A)\phi_0$. 
Then we have:
\begin{eqnarray*}
1 \times \lambda'&=&(1 \times s'^\vee)(1\times \lambda_{A'})(1\times s')=\\
&=&(1 \times (\phi_1^\vee \times \phi_1^\vee)s^\vee (\phi_2^{-1})^\vee )(\phi_1^\vee \times \phi_2^\vee) (1 \times \lambda_A) (\phi_1\times \phi_2)(1\times \phi_2^{-1}s(\phi_1\times \phi_1))=\\
&=&(\phi_1^\vee \times \phi_1^\vee \times \phi_1^\vee )(1 \times s^\vee )(1\times \lambda )(1\times s) (\phi_1 \times \phi_1 \times \phi_1)=\\
&=&(\phi_1^\vee \times \phi_1^\vee \times \phi_1^\vee)(1\times \lambda)(\phi_1\times \phi_1 \times \phi_1).
\end{eqnarray*}
By letting $\phi=\phi_1\times \phi_1 \times \phi_1$, it is straightforward to check that $\iota'(\theta)=\phi^{\vee}\iota(\,^\sigma \theta) \phi$.
\item [(b)] We consider now the case where $A$ and $A'$ are decomposable. 
If $\phi_0: =\begin{pmatrix} \phi_{1} & 0 \\ 0 & \phi_{2} \\ \end{pmatrix}:\,E'_0 \times A' \simeq E'_0 \times E'_1\times E'_2\rightarrow E\times A \simeq E_0 \times E_1 \times E_2$
then the proof is similar to the one in (a). 
Otherwise, assume that $\phi_{0}=\begin{pmatrix} 0 & \phi_{00} & 0\\ \phi_{01} & 0 & 0 \\ 0 & 0 & \phi_{02} \end{pmatrix}:\,E'\times A' \simeq E'_0\times E'_1\times E'_2\rightarrow E\times A \simeq E_0 \times E_1\times E_2$. Let $s'=\begin{pmatrix} s'_{11} & 0 \\ 0 & s'_{22} \\ \end{pmatrix}: E'_1 \times E'_2\rightarrow {E'}_{0}^{2}$ and $s= \begin{pmatrix} s_{11} & 0 \\ 0 & s_{22} \\ \end{pmatrix}: E_1\times E_2\rightarrow E_0^2$. 
Then it is easy to check that the homomorphism $\phi: \Hom^0(E'^3,E^3)$ that makes the following diagram commutative: 
\[ \begin{tikzcd}
E'_0\times E'_1 \times E'_2 \arrow{r}{\phi_0} \arrow[swap]{d}{1 \times s' } &  E_0\times E_1\times E_2 \arrow{d}{1 \times s} \\%
E'_0\times E'_0 \times E'_0 \arrow{r}{\phi}& E_0\times E_0 \times E_0
\end{tikzcd}
\]
is given by the formula 
\begin{eqnarray*}\phi=\begin{pmatrix} 0 & \phi_{00}s'^{\vee }_{11}/n & 0\\  s_{11}\phi_{01} & 0 & 0 \\ 0 & 0 & s_{22}\phi_{02} (s'_{22})^{\vee }/n \end{pmatrix},
\end{eqnarray*}
where the elliptic curves are identified with their duals. 
The condition on the polarization in Definition~\ref{def:eqREP} (2) follows from the fact that we have $1\times \lambda _{A'}=\phi_0^\vee (1 \times \lambda_A) \phi_0$ and that $s$ and $s'$ are such that $ \lambda = s ^{\vee} \lambda_A s$ and that $\lambda'=s'^{\vee}\lambda_{A'}s'$, respectively. In a similar manner, it is straightforward to check that $\iota'(\theta)=\phi^{-1} \iota(\,^{\sigma} \theta) \phi$. 
\end{itemize}
\end{proof}

\begin{remark}
{Note that since we were not able to proof that given a normal solution to the IEP, there is a solution to the REP which induces it, we cannot claim that the inverse implication of Proposition \ref{prop:IEP} is true. We leave this to future work.}
\end{remark}

The following two propositions give explicit formulae for equivalent solutions to the IEP obtained from equivalent solutions to the REP. 
These formulae will be used in Section~\ref{sa4} for counting equivalent solutions. 
First, let us rewrite condition (1) in Definition \ref{def:eqREP} for the corresponding \Int~solutions to the IEP. 

\begin{proposition}\label{prop: equivforExA}
Let $(E\times A, 1\times \lambda_A, \iota_0:\mathcal{O} \hookrightarrow \text{End}(E\times A))$ and $(E\times A, 1\times \lambda_A, \iota'_0: \mathcal{O} \hookrightarrow \text{End}(E\times A))$ be two solutions to the REP which are equivalent in the sense of Definition ~\ref{def:eqREP}(1) and let $(E,\eta,[x,d/n,a,c/n,b,\gamma,n,x_1,x_2,x_3])$  and $(E,\eta',[x',d'/n,a',c'/n,b',\gamma',n,x'_1,x'_2,x'_3])$  be the corresponding induced solutions to the IEP. Take $\mu=-\eta^2$ and $\mu'=-\eta'^2$ and assume that $\mu' = a_0+a_1\mu+a_2\mu^2$ for some $a_0,\,a_1$ and $a_2\in\Q$. 
Then we have:
\begin{align*}
x' =\, &a_0+a_1x+a_2x^2+a_2a,\\
a' =\, &a(a_2^2c/n+(a_1+a_2x)^2)+b(a_2^2d/n+2a_2(a_1+a_2x)),
\end{align*}
and $d'/n'$, $c'/n'$, $b'$, $\gamma'$, $n'$ are determined by Equations \eqref{relationsFromKLLNOS}-\eqref{eq:C}.

\end{proposition}

\begin{proof} Let $\iota_0(\mu) = \begin{pmatrix}x & y \\ z & w\end{pmatrix}$ as given in \eqref{eq:xyzw}. Since $\iota_0$ is a ring homomorphism, we have 
$$
\iota_0(\mu') = a_0\begin{pmatrix}1 & 0 \\ 0 & 1\end{pmatrix}+a_1\begin{pmatrix}x & y \\ z & w\end{pmatrix}+a_2\begin{pmatrix}x^2+yz & xy+yw\\ zx+wz & zy+ w^2\end{pmatrix}=\begin{pmatrix}x' & y'\\ z' & w'\end{pmatrix}.
$$ It follows from \eqref{eq:iotamu} that $x$ is an integer, $a=yz$, $b=ywz$ and $w^2z=z\frac{c}{n}+wz\frac{d}{n}$. Thus, we obtain $x' =a_0+a_1x+a_2x^2+a_2yz$, $y' =y(a_1+a_2x+a_2w)$ and $z' =(a_1+a_2x+a_2w)z$. 
Since $a'=y'z'$, we get the second equality in the statement.
\end{proof}

\begin{remark}\label{rmk:equivExA}
By similar arguments as in Proposition \ref{prop: equivforExA}, we could explicitly relate $x_1, x_2,x_3$ and $x_1', x_2', x_3'$ (and $d_1, d_2,d_3$ and $d_1', d_2', d_3'$ in the case that $K$ contains an imaginary quadratic field), but the formulas are overly complicated and we do not display them here. 
\end{remark}

We also rewrite the equations in Definition \ref{def:eqREP} (2) in terms of the corresponding equivalent \Int~solutions to the IEP.

\begin{proposition} \label{prop: equiv}
Let $(E\times A, 1\times \lambda_A, \iota_0: \mathcal{O} \hookrightarrow \End(E\times A))$ and  $(E'\times A', 1\times \lambda_{A'}, \iota'_0: \mathcal{O} \hookrightarrow \End(E'\times A'))$ be two solutions to the REP which are equivalent in the sense of Definition ~\ref{def:eqREP}(2).  Let $\eta\in\O$ totally imaginary and consider the two induced \Int~solutions  $(E,\eta,[x,d/n,b,c/n,a,\gamma,n])$ and $(E',\eta, [x',d'/n,b',c'/n,a',\gamma',n'])$ as in Section \ref{sec: embedding_img}. Then:
\begin{itemize}
\item[(1)] If $(A,\lambda_A)$ is indecomposable, then $[x,a,b,c,d,\gamma,n]=[x',a',b',c',d',\gamma',n']$ and $x_i=x'_i$ up to conjugation in $\mathcal{R}$.
\item[(2)] If $(A,\lambda_A)\simeq(E_1\times E_2,1\times 1)$, then {either $[x,a]=[x',a']$ or $[x',a']=[w_{11},\deg(z_1)+\deg(w_{12})]$ or $[x',a']=[w_{22},\deg(z_2)+\deg(w_{12})]$ where $x,y,z$ and $w=(w_{ij})_{1\leq i,j \leq 2}$ are as in Equation~\ref{eq:xyzw}.}
\end{itemize}
\end{proposition}
\begin{proof} 
Let  $\phi_0:\,E'\times A'\rightarrow E\times A$ an isomorphism of principally polarized abelian varieties such that $\iota'_0(\cdot)=\phi_0^{-1}\iota_0(\cdot)\phi_0$.
Let $\mu=-\eta^2$ and $\iota_0(\mu) = \begin{pmatrix}x & y \\ z & w\end{pmatrix}$  and let $\iota'_0(\mu) = \begin{pmatrix}x' & y' \\ z' & w'\end{pmatrix}$ as given in \eqref{eq:xyzw}. We have that
\begin{equation} \label{eq: REPexp}
\begin{aligned}
\iota'_0(\mu) &=  \phi_0^{-1} \iota_0(\mu) \phi_0  = \phi_0^{-1} \begin{pmatrix}
x & y\\
z & w
\end{pmatrix} \phi_0.
\end{aligned}
\end{equation}
\noindent
(1) If $(A,\lambda_A)$ is indecomposable, then $\phi_0 = \phi_{1}\times\phi_{2}$, where $\phi_{1} : E' \rightarrow E$ and $\phi_{2} :A'  \rightarrow A$ are isomorphisms. 
Then, by \eqref{eq: REPexp}, we have $x' = \phi^{-1}_{1}x\phi_{1} = x$, since $x \in \mathbb{Z}$, $y' = \phi^{-1}_{1}y \phi_{2}$, and $z' = \phi_{2}^{-1}z\phi_{1}$. Thus, $a'=y'z'= \phi^{-1}_{1}yz\phi_{1}=a$.
Since  $d'/n'$, $c'/n'$, $b'$, $\gamma'$, $n'$ satisfy Equations \eqref{relationsFromKLLNOS}-\eqref{eq:C} and depend only on $x'$ and $a'$, we have the equality in the proposition.

\noindent
(2) If $(A,\lambda_A)$ is decomposable, then it is isomorphic to $(E_1\times E_2,1\times 1)$. Use this isomorphism to identify them and write $y=(y_1,y_2)$ and $z=\begin{pmatrix}z_1\\z_2\end{pmatrix}$. We have $y_1=z_1^\vee$ and $y_2=z_2^\vee$. Moreover, $w=\begin{pmatrix}w_{11} & w_{12}\\ w_{12}^\vee & w_{22} \end{pmatrix}$ with $w_{11},w_{22}\in\Z$. The isomorphism $\phi_0$ has to be a permutation of the elliptic curves up to isomorphism, because it preserves the polarization. Let us consider first the case $\phi_{0}=\begin{pmatrix}\phi_{00} & 0 & 0\\ 0 & 0 & \phi_{01} \\ 0 & \phi_{02} & 0\end{pmatrix}:\,E'\times A' \simeq E\times E_2\times E_1\rightarrow E\times A \simeq E\times E_1\times E_2$, then as in the previous case it is straightforward to check that  $[x,d/n,a,c/n,b,\gamma,n]=[x',d'/n',a',c'/n',b',\gamma',n']$. 

Consider now $\phi_{0}=\begin{pmatrix}0 & \phi_{00} & 0 \\ \phi_{01} & 0 & 0 \\ 0 & 0 &  \phi_{02} \end{pmatrix}:\,E'\times A' \simeq E_1\times E\times E_2\rightarrow E\times A \simeq E\times E_1\times E_2$. Any other case will be a composition of cases of these two types. By {Equation}~\eqref{eq:iotamu},  we have that $a=\text{deg}(z_1)+\deg(z_2)$, $b=w_{11}\deg(z_1)+w_{22}\deg(z_2)+\Tr(z_1^\vee w_{12}z_2)$, $c/n=\deg(w_{12})-w_{11}w_{22}$ and $d/n=w_{11}+w_{22}$. By performing similar computations as in the other cases, we have that $x'=w_{11}$, $a'=\text{deg}(z_1)+\deg(w_{12})$, $b'=x\deg(z_1)+w_{22}\deg(z_2)+\Tr(z_1^\vee w_{12}z_2)$, $c'/n'=\deg(z_2)-xw_{22}$ and $d'/n'=x+w_{22}$.
\end{proof}

\begin{remark}\label{rmk: equiv} As in Proposition \ref{prop: equivforExA} we could explicitly relate $x_1, x_2,x_3$ and $x_1', x_2', x_3'$ 
(and $d_1, d_2,d_3$ and $d_1', d_2', d_3'$ in the case that $K$ contains an imaginary quadratic field) in Proposition \ref{prop: equiv}. But again the formulas are overly complicated and we do not display them here. 
\end{remark}

In addition to conditions (1) and (2) in Definition \ref{def:eqIEP}, when we define a \Int~solution to the IEP from a solution to the REP, we make another choice: {the isomorphisms $\mathcal{R}\simeq\End(E)$ and $\mathcal{R}'\simeq \End(E')$}. To take this into account,  we included condition (3) in Definition \ref{def:eqREP}. Automorphisms of $\mathcal{R}$ induce automorphisms of $\mathcal{B}=\mathcal{R}\otimes\mathbb{Q}$, and those are described by the Skolem-Noether Theorem below (see for instance \cite[Theorem 7.1.3, Corollaries 7.1.4 and 7.1.5]{Voight}).

\begin{proposition}\label{Prop:EquivAutom}
Every $\Q$-algebra automorphism of $\mathcal{B}$ is inner, i.e., $\Aut_\Q(\mathcal{B})\simeq \mathcal{B}^\times/\Q^{\times}$.
\end{proposition}

\begin{example} We revisit the example in Subsection \ref{sec:closeex}: we had $\mu^3 - 13\mu^2 + 50\mu - 49=0$, $i\in K$  and $A=E^2$ and the \Int~IEP and REP solutions :

$$
\iota(\mu)=\begin{pmatrix}
5 & 2 & 5\\
1 & 0 & -12\\
0 & 1 & 8
\end{pmatrix}\,\text{and}\,\,
\iota(i)=\begin{pmatrix}
i & 0& 0\\
0 & i & 0 \\
0 & 0 & i\end{pmatrix},
$$

$$
\iota_0(\mu)=\begin{pmatrix}
5 & 1 & 1\\
1 & 4 & \frac{-3+j}{2}\\
1 & \frac{-3-j}{2} & 4
\end{pmatrix}
\,\text{and}\,\,\iota_0(i)=\begin{pmatrix}
i & 0& 0\\
0 & 0 & i \\
0 & i &
 0\end{pmatrix}.
$$
Take $E'\times A'=E\times E\times E$ where we permute the first and second copy of the elliptic curve. This gives the equivalent solution to the REP (according to Definition \ref{def:eqREP}(2)):
$$
\iota'_0(\mu)=\begin{pmatrix}
4 & 1 & \frac{-3+j}{2}\\
1 & 5 & 1\\
\frac{-3-j}{2} & 1 & 4
\end{pmatrix}
\,\text{and}\,\,\iota'_0(i)=\begin{pmatrix}
0 & 0& i\\
0 & i &0 \\
i & 0 &
 0\end{pmatrix},
$$
which produces the equivalent \Int~IEP solution (accordingly to Definition \ref{def:eqIEP}(2)):
$$
\iota'(\mu)=\begin{pmatrix}
4 & 5 & 18\\
1 & 0 & -19\\
0 & 1 & 9
\end{pmatrix}\,\text{and}\,\,
\iota'(i)=\begin{pmatrix}
i & 0& 0\\
0 & i & 0 \\
0 & 0 & i\end{pmatrix}.
$$

Starting again by the REP solution $\iota$ we could have taken another conjugate of $\mu$, for example $\mu'=\mu^2 - 8\mu + 16$ to get  the equivalent solution to the REP (according to Definition \ref{def:eqREP}(1)):
$$
\iota'_0(\mu)=\iota_0(\mu')=\begin{pmatrix}
3 & \frac{-1-j}{2}& \frac{-1+j}{2}\\
\frac{-1+j}{2} & 5 & 1\\
\frac{-1-j}{2} & 1 & 5
\end{pmatrix}
\,\text{and}\,\,\iota'_0(i)=\begin{pmatrix}
i & 0& 0\\
0 & 0 & i \\
0 & i &
 0\end{pmatrix},
$$
that yields the equivalent \Int~IEP solution (accordingly to Definition \ref{def:eqIEP}(1)):
$$
\iota'(\mu)=\begin{pmatrix}
3 & 4 & 17\\
1 & 0 & -24\\
0 & 1 & 10
\end{pmatrix}\,\text{and}\,\,
\iota'(i)=\begin{pmatrix}
i & 0& 0\\
0 & i & 0 \\
0 & 0 & i\end{pmatrix}.
$$
\end{example}

%% file: Sec6-CountingSolutions.tex

In this section we compute the number of \Int~solutions
\begin{itemize}
    \item $(E,\eta,[x,d/n,a,c/n,b,\gamma,n,x_1,x_2,x_3])$ to the IEP when $K$ does not contain an imaginary quadratic field, where $\eta$ is a root of the polynomial $t^6+At^4+Bt^2+C=0$;
    \item $(E,\mu, f\sqrt{D}, [x,d/n,a,c/n,b,\gamma,n,d_1,d_2,d_3])$ to the IEP when $K$ contains an imaginary quadratic field $\Q(\sqrt{D})$, where $\mu$ is a root of the polynomial $t^3-At^2+Bt-C=0$, and $D\in \Z_{<0}$ is square-free, and $f \in \Z$.
\end{itemize}

We use the notation introduced in Section \ref{sec: embedding_img}. First we will describe some bounds for $[x,d/n,a,c/n,b,\gamma,n]$, $[x_1,x_2,x_3]$ and $[d_1,d_2,d_3]$ in terms of $A,B$ and $C$. Then, in Section \ref{Ss:BoundForp}, we describe the bounds for the primes $p$ for which there are solutions. This improves the bounds in \cite{KLLNOS}. In the case when the field $K$ contains an imaginary quadratic field, we do not only give a bound but also a very small list of primes for which we can find solutions. Finally, Section \ref{sec: MainAlgorithm} contains the algorithm and some implementation details.

%% file: Sec6.1-Bounds.tex

\subsection{Bounds for $x,d/n,a,c/n,b,\gamma,n$}\label{Ss:Bounds}
\mbox{}
Using \eqref{eq:A}--\eqref{eq:C}, we obtain expressions for $\frac{d}{n}, \frac{c}{n}, b$ in terms of $A,B,C,x,a$:
$$
    \frac{d}{n}  = - x +A, \qquad    \frac{c}{n}  = - x^2 + Ax - a - B , \qquad    b  = - x^3 + Ax^2 - 2ax - Bx + Aa + C. 
$$
Since $\gamma$ and $n$ are determined by \eqref{eq1:gamma} and \eqref{eq1:n}, this means that, given $A, B, C$, it is enough to find all the possibilities for $x$ and $a$ to determine all possible values for $[x,d/n,a,c/n,b,\gamma,n]$.

\begin{lemma}  \label{bd: x and a}
We have $0< x < \sqrt{A^2 - 2B}$, $0 < a \leq \frac{A^2 - 2B - x^2}{2}$, and $\gamma \leq a\left(A^2 -2B -x^2 - 2a \right)$.
\end{lemma}

\begin{proof} Recall from the text below (\ref{eq:polarization}) that we have $\gamma, a, n >0$.
We first show that $x>0$. Suppose $x\leq 0$. By \eqref{eq:A}, we have $A = \frac{d}{n} +x >0$. Thus, $\frac{d}{n} > 0$. Then, $n>0$ implies $d > 0$. Moreover, from \eqref{eq:B} we have $B + \frac{c}{n} = \frac{dx}{n} - a < 0$. Since $B>0$, we obtain $c<0$. Now \eqref{eq:C} implies that $b = C+ \frac{cx}{n} + \frac{ad}{n} > 0$. Then,
$$
0< n = a\gamma - b^2 = \frac{a^2c}{n} + \left(\frac{ad}{n} - b \right)b < \frac{a^2c}{n} - \frac{cx}{n}b < 0,
$$
which leads to a contradiction. Hence $x>0$.

By combining \eqref{eq1:n}--\eqref{eq:B}, we can rewrite $\Tr_{K_+/\Q}(\mu^2)=A^2-2B$ as 
\begin{equation*}
A^2-2B = x^2+2a+2\frac{c}{n}+\frac{d^2}{n^2} = x^2+2a+2\frac{\gamma}{a}+\left(\frac{b}{a}-\frac{d}{n}\right)^2-\frac{b^2}{a^2} = x^2+2a+\frac{\gamma}{a}+\left(\frac{b}{a}-\frac{d}{n}\right)^2+\frac{n}{a^2} >0.
\end{equation*}
Since $\gamma, a, n >0$, we obtain the claimed bounds. 
\end{proof}

%% file: Sec6.2-BoundsForNorms.tex

\subsection{Bounds for traces and norms} \label{sec: bounds_for_xi} %

The bounds in this section are used in~\textbf{Algorithm \ref{sa2}}. We first remark the following.

\begin{remark}\label{remark:integers}
Let $y_i, y_j \in \mathcal{B}_{p, \infty}$ such that $\Tr(y_i) = \Tr(y_j) = 0$, and $r,s\in \Q$. We have 
\begin{equation}\label{eq: Trxixj-int}
rs\Tr(y_iy_j) = \N(sy_i - ry_j) - s^2\N(y_i) - r^2\N(y_j).
\end{equation}
Thus, since $y_i - y_j \in \mathcal{B}_{p, \infty}$, we have that $\N(y_i), \N(y_i), \N(y_i - y_j) \in \mathbb{Z}_{\geq 0}$. Hence, by taking $r=s=1$, we obtain $\Tr(y_iy_j) \in \mathbb{Z}$.
\end{remark}

\begin{lemma}\label{Lemma:bdTrace} For every $y_i, y_j \in \mathcal{B}_{p, \infty}$, we have
$$
\mid \Tr(y_iy_j)\mid \leq 2(\N(y_i)\N(y_j))^{1/2}.
$$
\end{lemma}
\begin{proof} This is just a straightforward application of Cauchy-Schwarz inequality.
\end{proof}

Combining Lemma \ref{Lemma:bdTrace} with the bounds for the norms given in Lemmata \ref{bd: Nx} and \ref{bd: Nd} below leads to an efficient algorithm to compute all possible solutions. First notice the following.

\begin{remark}\label{rmk: NeNf}
For $e,e' \in \mathcal{B}_{p, \infty}$, we have $\N(e+e')+\N(e-e')=2(\N(e)+\N(e'))$. This implies 
$$
\N(e + e') \leq 2(\N(e)+\N(e')).
$$
\end{remark}

\begin{lemma}  \label{bd: Nx}
With the notation as in Proposition \ref{prop: lambda2} we have
\begin{align*}
\N(x_1) &\leq x,\\ 
\N(x_2) &\leq \min\left\{a\left( x-\N(x_1) \right), \frac{2}{\gamma} \left( n \left( x-\N(x_1) \right) + \frac{b^2-n}{\gamma}\N(x_3) \right) \right\}, \\
\N(x_3) &\leq \min \left\{ \gamma(x - \N(x_1)), \frac{2}{a} \left( n \left( x-\N(x_1) \right) + \frac{b^2-n}{a}\N(x_2) \right)
 \right\}. 
 \end{align*}
\end{lemma}

\begin{proof}
By \eqref{eq: Trxixj-int}, the first equation in Proposition \ref{prop: relations} can be written in the following two ways
\begin{equation}\label{rewrite:x}
x =\N(x_1) + \frac{n}{\gamma}\N\left(\frac{\gamma}{n}x_2 - \frac{b}{n}x_3\right) + \frac{1}{\gamma}\N(x_3)=\N(x_1) + \frac{n}{a}\N\left(\frac{b}{n}x_2 - \frac{a}{n}x_3\right) + \frac{1}{a}\N(x_2).
\end{equation}
Using the first part of \eqref{rewrite:x}, we obtain $\N(x_1) \leq x$, $\N(x_3) \leq \gamma\left( x-\N(x_1) \right)$, and
\begin{align*} 
\N\left(\frac{\gamma}{n}x_2 - \frac{b}{n}x_3\right) &\leq \frac{ \gamma \left( x-\N(x_1) \right) - \N(x_3) }{n}. 
\end{align*}
By taking $e = \frac{\gamma}{n}x_2 - \frac{b}{n}x_3$ and $e' = \frac{b}{n}x_3$ in Remark \ref{rmk: NeNf}, we get
$$
 \frac{\gamma^2}{n^2}\N(x_2) \leq 2\left( \N\left(\frac{\gamma}{n}x_2 - \frac{b}{n}x_3\right) + \frac{b^2}{n^2}\N(x_3) \right) \leq  2\left(\frac{ \gamma \left( x-\N(x_1) \right)}{n} + \frac{b^2-n}{n^2}\N(x_3) \right).
$$
In a similar manner, the second part of \eqref{rewrite:x} leads to the claimed bounds.
\end{proof}

\begin{lemma}\label{bd: Nd} If $f\sqrt{D} \in \O$ with $D \in \Z_{<0}$ then, with the notation as in Proposition \ref{prop: lambda3}, we have  
\begin{align*}
\N(d_1) &\leq -f^2D, \\
\N(d_2) &\leq \min\left\{a\left(-f^2D-\N(d_1) \right), \frac{2}{\gamma} \left( n \left(-f^2D -\N(d_1) \right) + \frac{b^2-n}{\gamma}\N(d_3) \right) \right\}, \\
\N(d_3) &\leq \min \left\{ \gamma(-f^2D - \N(d_1)), \frac{2}{a} \left( n \left(-f^2D -\N(d_1) \right) + \frac{b^2-n}{a}\N(d_2) \right)
 \right\}. 
 \end{align*}
 \end{lemma}
\begin{proof}
The proof is similar to the one of Lemma \ref{bd: Nx}. The bounds follow after rewriting the first equation in Proposition \ref{prop: relations_d}.
\end{proof}

%% file: Sec6.3-BoundsForPrimes.tex

\subsection{Bounds for primes}\label{Ss:BoundForp}

In this section we will present bounds for the primes of bad reduction for genus 3 curves with CM by an order $\O$ in a sextic CM field. We give bounds for the following three different cases: 
\begin{itemize}
\item[(i)] If $\O$ contains $\zeta_4$, then $p\leq (A^2-2B)^3/2$ or $p\mid 2n$, 
\item[(ii)] If $\O$ contains $f\sqrt{D}$ with $D \in \Z_{<0}$ and $f \in \Z$, then $p\leq f^4D^2(A^2-2B)^3$, 
\item[(iii)] If $K$ does not contain an imaginary quadratic field, then $p\leq (A^2-2B)^4$,
\end{itemize}
where $A = \Tr_{K_+/\Q}(\mu)$ for a totally positive real $\mu \in \O$ such that $K_+ = \Q(\mu)$. We take $f := [\O_{\Q(\sqrt{D})} : \O\cap \Q(\sqrt{D})]$ so that $f\sqrt{D} \in \O$.

The bounds obtained in this section are not only asymptotically better than the bounds given in \cite{BCLLMNO,KLLNOS}, but Theorem \ref{thm: PrimeBdQ(i)} gives also a very short list of primes for case (i). Notice that $A=\Tr_{K_+/\Q}(\mu)$ in this paper corresponds to what is called $B$ in \cite{KLLNOS}, so while the bound in case (iii) in there is $O(A^{10})$, our bound is $O(A^8)$. 
 Checking that our bound for case (ii) is (actually much) better is more subtle: in this case, the roots of the sextic polynomial are of the form  $\pm\alpha_1\sqrt{d}$, $\pm\alpha_2\sqrt{d}$ and $\pm\alpha_3\sqrt{d}$ with $d=f^2D$, so the bound in \cite{KLLNOS} is $O(B^{10})=O(d^{10}\alpha_i^{20})$ and ours for the same polynomial is $O(d^8\alpha_i^{12})$, but in our setting we can directly work with the cubic polynomial defining the $\alpha_i$, so we get a bound $O(d^2\alpha_i^6)$.

\begin{theorem}[Case (i)] \label{thm: PrimeBdQ(i)} Let $K$ be a sextic CM field. Let $\O$ be an order in a sextic CM field containing~$\zeta_4$.
Let $X/M$ be a genus-3 curve such that its Jacobian is absolutely simple and has CM by $\O$. Let $\mathfrak{p}\subset \O_M$ 
(lying above rational prime $p$) be a prime of bad reduction for $X$. 
Then we have $p\mid 2fn$, where $n$ is a positive integer such that $[n]\ker(s) = 0$ in Lemma \ref{isogeny}. In particular, we have $p\leq  (A^2-2B)^3/2$ or $p\mid 2n$.
\end{theorem}
\begin{proof} 
By Lemma \ref{bd: x and a}, we obtain 
$$
n=a\gamma - b^2 \leq a^2\left(A^2 -2B -x^2 - 2a \right)  \leq \frac{\left(A^2 - 2B - x^2\right)^3}{2}. 
$$
Since, $x > 0$, we get the claimed bound.

By Corollary \ref{cor: lambda3}, $\iota(\sqrt{-1})$ is a diagonal matrix. Then, by Proposition \ref{prop: primitivity}, we get $p \mid 2fn$. In this case $f=1$ because $\zeta_4\in \O$ hence $\O\cap \Z[\zeta_4] = \Z[\zeta_4]$ So we only have $p \mid 2n$.
\end{proof}

\begin{remark}\label{remark: PrimeBdQ(zeta7)} 
The previous theorem is analogous for sextic orders containing $\zeta_7$. If a sextic order contains $\zeta_7$, then $K = \Q(\zeta_7)$ and the order containing $\zeta_7$ is the maximal order $\Z[\zeta_7]$. 
It is shown in \cite[Section 5.2]{IKLLMMV_mod} that there is 
only one curve with CM by $\Z[\zeta_7]$, and this curve does not have bad reduction.  
\end{remark}

To prove the cases (ii) and (iii) we follow the classical approach used in \cite{GorenLauter07, KLLNOS}. nevertheless, as already mentioned, the bounds presented in this section are better than the ones in the aforementioned references.

The main ingredient of the proofs of Theorems \ref{thm: PrimeBdImagQuad} and \ref{thm: PrimeBdgen} is the following lemma.

\begin{lemma}\cite[Corollary 2.1.2]{GorenLauter07} \label{lemma: comm}
Let $\calR$ be an order in a quaternion algebra $\calB$ and let $x,y\in \calR$. If $\N(x)\N(y) < p/4$, then $xy = yx$. 
\end{lemma}

\begin{theorem}[Case (ii)] \label{thm: PrimeBdImagQuad} Let $K$ be a sextic CM field containing $k = \Q(\sqrt{D})$ where where $D\in \Z_{<0}$ is square-free. Let $X/M$ be a genus-3 curve such that its Jacobian is absolutely simple and has CM by an order $\O$ in $K$. Let $f$ be the conductor of $\OO \cap k$ in $\OO_k$. 
Let $\mathfrak{p}\subset \O_M$ (lying above a rational prime $p$) be a prime of bad reduction for $X$. Then we have $p \leq f^4D^2(A^2-2B)^3$.
\end{theorem}
\begin{proof} By Proposition \ref{prop: embeddingR}, bad reduction gives a solution to the REP. This gives a solution to the IEP given by $\iota(\mu)=\Lambda_1$ and 
$\iota(f\sqrt{D})=\Lambda_3$. 
The entries of $\Lambda_1$ are rationals and the entries of $\Lambda_3$ are in $\Q(d_1,d_2,d_3)$, where $d_1,d_2,d_3\in\mathcal{R}$. Using the bounds in Lemma \ref{bd: Nd}, we have $\N(d_1)<- f^2D$, $\N(d_2)\leq -af^2D$ and $\N(d_3)\leq -\gamma f^2D$. Thus, by Lemma \ref{bd: x and a}, for any $d_i,d_j \in \{ d_1, d_2, d_3\}$, we have 
$$
4\N(d_i)\N(d_j)\leq 4f^4D^2a\gamma \leq 4f^4D^2a^2 \left(A^2 - 2B - x^2 -2a\right) \leq f^4D^2 \left(A^2 - 2B - x^2\right)^3 \leq f^4D^2 \left(A^2 - 2B\right)^3.
$$
Therefore, if $p> f^4D^2 \left(A^2 - 2B\right)^3$, by Lemma \ref{lemma: comm}, the quaternions $d_i$ commute. This implies that $\iota(f\sqrt{D}) = \text{diag}(f\sqrt{D}, f\sqrt{D}, f\sqrt{D})$. Hence, by Proposition~\ref{prop: primitivity}, we have $p \mid 2f D n$. As in Theorem~\ref{thm: PrimeBdQ(i)}, we get $p \leq fD\left(A^2 - 2B\right)^3/4$. If $p \nmid 2f D n$, then commutativity of $d_i$'s gives a contradiction. Therefore, we obtain that
$$
p \leq \max\{f^4D^2 \left(A^2 - 2B\right)^3, fD\left(A^2 - 2B\right)^3/4\} \leq f^4D^2 \left(A^2 - 2B\right)^3
$$
as claimed.
\end{proof}

\begin{theorem}[Case (iii)] \label{thm: PrimeBdgen} Let $K$ be a sextic CM field not containing an imaginary quadratic field. Let $X/M$ be a genus-3 curve such that its Jacobian is absolutely simple and has CM by an order $\O$ in $K$. Let $\mathfrak{p} \subset \O_M$ (lying above a rational prime $p$) be a prime of bad reduction.
Then we have $p\leq \left(A^2 - 2B\right)^4$. %
\end{theorem}
\begin{proof} By Proposition \ref{prop: embeddingR}, bad reduction gives a solution to the REP. This gives a solution to the IEP given by $\iota(\mu)=\Lambda_1$ and $\iota(\eta)=\Lambda_2$. The entries of $\Lambda_1$ are rationals, and the entries of $\Lambda_2$ are in $\Q(x_1,x_2,x_3)$, where $x_1,x_2,x_3\in\mathcal{R}$. Using the bounds in Lemma \ref{bd: Nx}, we have $\N(x_1)<x$, $\N(x_2)\leq ax$, and $\N(x_3)\leq \gamma x$. By Lemma \ref{bd: x and a}, for any $x_i,x_j \in \{x_1, x_2, x_3\}$, we obtain 
$$
4\N(x_i)\N(x_j) \leq 4 a\gamma x^2 \leq 4x^2 a^2 \left(A^2 - 2B - x^2 -2a\right) \leq x^2 \left(A^2 - 2B - x^2\right)^3 \leq \left(A^2 - 2B\right)^4.
$$ 
Therefore, if $p > \left(A^2 - 2B\right)^4$, by Lemma \ref{lemma: comm}, the quaternions $x_i$ commute. Thus, 
by Proposition~\ref{prop: xi_non-comm}, the field $K$ contains an imaginary quadratic field, which leads to a contradiction.
\end{proof}

\begin{remark} \label{remark: PrimeBdgen}
In practice, we use the much better bounds $4a\gamma f^4D^2$ and $p \leq 4a\gamma x^2$, respectively, which follow from the bounds in Lemma \ref{bd: Nx}, as stated in the proofs of Theorems \ref{thm: PrimeBdImagQuad} and \ref{thm: PrimeBdgen}. Moreover, before we apply \textbf{Algorithm \ref{Malg}} in Section~\ref{sec: MainAlgorithm}, we first eliminate potentially good primes using Corollary~\ref{cor: pot_good_primes}.
\end{remark}

%% file: Sec6.5-Lifting.tex

\subsection{Lifting arguments}\label{sec: lifting arguments}

Given a tuple $[x,d/n,a,c/n,b,\gamma,n]$ satisfying 
 \begin{align}\label{eq:L1}
  \begin{aligned}
   0 & < x  < \sqrt{A^2 - 2B}  \\       \frac{c}{n} &= \frac{d}{n}x - a -B
  \end{aligned}
  &&
  \begin{aligned}
   0 & < a  \leq \frac{A^2 - 2B - x^2}{2} \\       b &= \frac{c}{n}x+\frac{d}{n}a +C
  \end{aligned}
  &&
  \begin{aligned}
   \frac{d}{n} &= A-x \\       n &= a\gamma - b^2,
  \end{aligned}
 \end{align}
where $\gamma = a\frac{c}{n} + b\frac{d}{n}$,
 we obtain an embedding $\iota_\mu:\,\mathbb{Z}[\mu]\hookrightarrow\calM_{3 \times 3}(\calR)$ as in \eqref{eq:iotamu}. We say that \textit{$\iota_\mu$ lifts to $\O_+$} if this embedding extends to an embedding of $\O_{+}$ into $\calM_{3 \times 3}(\mathcal{R}/n)$ as in Lemma~\ref{prop: embeddingI}. %
More specifically, we write a $\Z$ basis for $\O_+$ in terms of the basis elements $1, \mu, \mu^2$ of $\Z[\mu]$, say 1, $u_1$, $u_2$ is such a basis for $\O_+$. Then $\iota_\mu$ lifts to $\O_+$ if $\iota(u_1)$ and $\iota(u_2)$ are matrices in $\calM_{3 \times 3}(\mathcal{R}/n)$ with the entries of the top row in $\calR$.

Similarly, for a given a tuple $[x,d/n,a,c/n,b,\gamma,n, x_1, x_2, x_3]$ satisfying \eqref{eq:L1} and Proposition~\ref{prop: relations}, we say that $\iota$ \textit{lifts to} $\O$ if for a given $\Z$ basis $1, v_1, v_2, v_3, v_4, v_5$ of $\O$ (which is written in terms of $1, \eta, \eta^2, \eta^3, \eta^4, \eta^5$) we have $\iota(v_i) \in \calM_{3 \times 3}(\mathcal{R}/n)$ for all $i$ with the entries of the top row in $\calR$.

Lastly, in case $K$ contains an imaginary quadratic field $\Q(\sqrt{-D})$ we also have the following lift. For a given a tuple $[x,d/n,a,c/n,b,\gamma,n, d_1, d_2, d_3]$ satisfying \eqref{eq:L1} and Proposition \ref{prop: relations_d}, we say that $\iota$\textit{ lifts to} $\O \cap \Z[\sqrt{-D}]$ if for a given $\Z$ basis $1, w$ of $\O \cap \Z[\sqrt{-D}]$ (which is written in terms of $1, \sqrt{-D}$) we have $\iota(w) \in \calM_{3 \times 3}(\mathcal{R}/n)$ with the entries of the top row in $\calR$.

\subsubsection{Optimality}
When the index of $\mathcal{O}$ in $\mathcal{O}_K$ is not a power of $p$ we also need to check that the embedding~$\iota_0$ is optimal, see Proposition \ref{prop: embeddingR}.

%% file: Sec6.6-Algorithm.tex
\subsection{Formulas à la Lauter-Viray}

Following the idea in \cite[Proposition 3.1]{LV15b}, we compute the discriminant of an order contained in $\calB_{p, \infty}$. Since $p$ ramifies in a quaternion algebra $\calB_{p,\infty}$, it ramifies in any order contained in $\calB_{p,\infty}$. We denote by $\D(y)$ the discriminant of an element $y \in \calB_{p,\infty}$. In particular, we have 
$$\D(y) := \Tr(y)^2 - 4\N(y) \in \Z.$$
The following lemma is used to eliminate some outputs in Steps (5), (7), (9) of \textbf{Algorithm~\ref{sa2}}. 

\begin{lemma}\label{Lemma:calR} Let $y_1, y_2, y_3 \in \calB_{p,\infty}$ such that $\Tr(y_i)=0$ with $i \in \{1,2,3\}$. Let $\calR_{123} = \Z + \Z y_1 + \Z y_2 + \Z y_3$ and $\calR_{ij} = \Z + \Z y_i + \Z y_j + \Z y_iy_j^\vee$, where $y_i, y_j \in \{y_1, y_2, y_3\}$. Then $\disc(\calR_{ij}) =  \left(\D(y_iy_j)\right)^2$ and 
\begin{align}\nonumber
\disc(\calR_{123})  
& = 4\left ( 4\N(y_1)\N(y_2)\N(y_3) - \N(y_1)\left(\Tr(y_2y_3)\right)^2 - \N(y_2)\left(\Tr(y_1y_3)\right)^2 \right.\\  
&\left. -  \N(y_3)\left(\Tr(y_1y_2)\right)^2 - \Tr(y_1y_2) \Tr(y_1y_3)\Tr(y_2y_3) \right)\nonumber \\
& = \left(2 \Tr(y_1 y_2 y_3) \right)^2.\nonumber
\end{align}
\end{lemma}

\begin{proof}
Both results follow from the definition of discriminant of an order in $\calB_{p,\infty}$.
\end{proof}

\subsection{The algorithm}\label{sec: MainAlgorithm}

Following the previous discussion, we now state our Main Algorithm to compute solutions to the IEP. We take $\mathcal{O}$ an order in a sextic CM field $K$, $p$ a prime number and $\mathcal{R}$ a maximal order in the quaternion algebra $\mathcal{B}_{p,\infty}$.

\RestyleAlgo{ruled}
\LinesNumbered
\begin{algorithm}[ht]
  \caption{Main Algorithm\label{Malg}}
 \KwData{$\mathcal{O}$, $p$, $\mathcal{R}$}
 \KwResult{All \Int~solutions $\mathcal{O}\hookrightarrow\mathcal{M}_{3 \times 3}(\mathcal{R}/n)$ to the IEP up to equivalence.}
Compute $A,B,C$ as in Section \ref{ABC};

Compute all $[x,d/n,a,c/n,b,\gamma,n]$ with \textbf{Algorithm \ref{sa1}};

Eliminate solutions $[x,d/n,a,c/n,b,\gamma,n]$ which do not lift to $\O_+$ or are not optimal, using the arguments in~Section~\ref{sec: lifting arguments}; 

For each remaining $[x,d/n,a,c/n,b,\gamma,n]$, compute all $[\N(x_1), \N(x_2),\N(x_3),\Tr(x_1x_2),\Tr(x_1x_3),\Tr(x_2x_3)]$ with \textbf{Algorithm \ref{sa2}};

For all $\N(x_1), \N(x_2), \N(x_3)$ appearing in the solutions $[x,d/n,a,c/n,b,\gamma,n, \N(x_1), \N(x_2),\N(x_3),\Tr(x_1x_2),\Tr(x_1x_3),\Tr(x_2x_3)]$ computed above, find all $x_1, x_2, x_3 \in \mathcal{R}$ with trace 0 and specified norms; 

For each $[x,d/n,a,c/n,b,\gamma,n, \N(x_1), \N(x_2),\N(x_3),\Tr(x_1x_2),\Tr(x_1x_3),\Tr(x_2x_3)]$, use \textbf{Algorithm \ref{sa3}}, to find which $x_1, x_2, x_3$ found above also have the specified $\Tr(x_1x_2),\Tr(x_1x_3),\Tr(x_2x_3)$;

Eliminate solutions $[x,d/n,a,c/n,b,\gamma,n, x_1, x_2, x_3]$ which do not lift to $\O$ or are not optimal, using the arguments in~Section~\ref{sec: lifting arguments};

Check equivalence of solutions due to Proposition \ref{prop: equivforExA} with \textbf{Algorithm~\ref{sa4}};

Check equivalence of solutions due to Proposition \ref{Prop:EquivAutom}.
\end{algorithm}

\begin{remark} \label{remark: imaginary} 
If the CM field $K$ contains an imaginary quadratic field $\Q(\sqrt{D})$, then we may embed the cubic order $\OO \cap K_+$ and the quadratic order $\OO \cap \Q(\sqrt{D})$ instead of embedding the sextic $\OO$ order to get a \Int~IEP solution as in Section~\ref{sec:iqf}, i.e. we can describe the \Int~ solutions as in Remark \ref{remark: sol_imag}. This computation is more efficient since the norm and trace bounds are smaller for the coefficients $d_1, d_2, d_3$ of $\Lambda_3$ than the coefficients $x_1, x_2, x_3$ of $\Lambda_2$. The algorithm for this case is exactly the same as \textbf{Algorithm~\ref{Malg}}, except that in Steps 4--6 and 8, 9 we consider $d_1, d_2, d_3$ instead of $x_1, x_2, x_3$. In Step 7, we eliminate solutions which do not lift to $\OO\cap \Q(\sqrt{D})=\mathbb{Z}[f\omega]$. 
\end{remark}

We now develop each step:

\subsubsection{The triples $(A,B,C)$} \label{ABC}

\begin{enumerate}
\item If $K$ does not contain an imaginary quadratic subfield, then take a totally imaginary $\eta \in \O$ such that $\mu:=\eta^2$ is totally positive with $K = \Q(\eta)$ and compute its minimal polynomial $t^6 + At^4 + Bt^2 + C$.
\item If $K$ contains $k = \Q(\sqrt{D})$ with $D \in \Z_{<0}$, then take a totally positive element $\mu \in \O-\Z$ and compute the minimal polynomial $t^3 - At^2 + Bt - C$ of $\mu$.
\end{enumerate}
Note that we choose $\eta$ and $\mu$ such that the index $[\O_+ : \Z[\mu]]$ is small enough and, if possible, such that $p \nmid [\O_+ : \Z[\mu]]$. When the index gets larger, the number of IEP solutions increases and the algorithm spends more time to eliminate non-\Int~solutions. When $p \nmid [\O_+ : \Z[\mu]]$, we can invoke Proposition~\ref{prop: primitivity} and eliminate many solutions in the existence of the diagonal matrix. 

\subsubsection{The tuples $[x,d/n,a,c/n,b,\gamma,n]$}
Given $A,B$ and $C$, we determine all the possibilities for $\iota_0(\mu)$ for a \Int~solution.

\RestyleAlgo{ruled}
\LinesNumbered
\SetAlgoNoLine
\begin{algorithm}[H]
\caption{\label{sa1}}
  
 \KwData{$A,B,C$}
 \KwResult{ All $[x,d/n,a,c/n,b,\gamma,n]$ satisfying equations \eqref{eq:A}--\eqref{eq:C} and Lemma \ref{bd: x and a}.}
For {$x\in[1, \lfloor \sqrt{A^2-2B} \rfloor ]\cap\mathbb{Z}$;}{

	 Compute $\frac{d}{n} = A-x$ from \eqref{eq:A};

	For $a \in [1, \lfloor \frac{A^2-2B-x^2}{2} \rfloor ] \cap\mathbb{Z}$;

	Compute $\frac{c}{n} = \frac{d}{n}x - a -B$ from \eqref{eq:B};

	Compute $b = \frac{c}{n}x+\frac{d}{n}a +C$ from \eqref{eq:C};

	Compute $\gamma = \frac{c}{n}a+\frac{d}{n}b$ and check if $\gamma>b^2/a$;

	 Compute $n=a \gamma - b^2$.}
\end{algorithm}

 The complexity of the algorithm mainly depends on the size of $A=-\Tr_{K_+/\Q}(\mu)$.

\subsubsection{Norms and traces}\label{SS:NT} Given a certain positive integer $N$, a prime $p$, a quaternion order~$\calR$ given by an integral basis $B = M\begin{pmatrix} 1 & i & j & k \end{pmatrix}^T$, where $M$ is a $4 \times 4$ matrix with rational entries, then an element $y \in \calR$ can be written as $y = u b_{1} +  v b_{2} + w b_{3} + t b_{4}$, where $b_{r}$ is the $r$-th term in $B$ and $u, v,w,t \in \mathbb{Z}$. If $\Tr(y)=0$, then one of $u,v,w,t$, can be expressed in terms of the other three variables. Then, $\N(x)$ is a polynomial in the remaining three variables. Running through all integers modulo $p$ for each of the three variables, we can find all the values that $\N(y)$ can achieve modulo $p$. If $N$ belongs to this set of values modulo $p$, then we say that $N$ is \emph{achievable}. 

The system of equations in Proposition \ref{prop: relations} is rewritten in the following way which is convenient for the algorithm.
\begin{align}
\N(x_3) &=  \left(\frac{ac}{n} + bx\right)\left(x -\N(x_1)\right) + ab - \frac{c}{n}\N(x_2) + b\Tr(x_1x_2)\label{x_alg:eq1}\\
\Tr(x_2x_3) &= \left(\frac{ad}{n}  - ax - b\right)\left(x-\N(x_1)\right) - a^2 - \frac{d}{n} \N(x_2) - a\Tr(x_1x_2)\label{x_alg:eq2}\\
\Tr(x_1x_3) &= \left(\frac{d}{n}x - x^2 - a + \frac{c}{n}\right)\left(x-\N(x_1)\right) + \frac{ad}{n} - ax  - b + \N(x_2) + \left(\frac{d}{n}  - x\right)\Tr(x_1x_2)\label{x_alg:eq3}.
\end{align}

\RestyleAlgo{ruled}
\LinesNumbered
\begin{algorithm}
  \caption{\label{sa2}}
 \KwData{$p$ and $[x,d/n,a,c/n,b,\gamma,n]$}
 \KwResult{ All $[ x,d/n,a,c/n,b,\gamma,n, \N(x_1), \N(x_2), \N(x_3),  \Tr(x_1x_2),  \Tr(x_1x_3),  \Tr(x_2x_3) ]$ satisfying equations \eqref{x_alg:eq1} - \eqref{x_alg:eq3} and bounds in Lemmas \ref{bd: Nx}, \ref{Lemma:bdTrace} and \ref{Lemma:calR}.}
 For each achievable $\N(x_1) \in [0, x]$;

 For each achievable $\N(x_2) \in  \left[0, a \left(x-\N(x_1)\right)\right]$;
	
 For each $\Tr(x_1x_2) \in \left[ -\lfloor{2(\N(x_1)\N(x_2))^{1/2}}\rfloor, \lceil{2(\N(x_1)\N(x_2)^{1/2}}\rceil \right]$;
	
 If $ \N(x_1) = \Tr(x_1x_2) =0 $ or $ \N(x_2) = \Tr(x_1x_2) =0$ or $ \N(x_1)\N(x_2) \neq 0 $, compute $\D(x_1x_2)$;

 If $p \mid \D(x_1x_2)$, compute $\N(x_3)$ using \eqref{x_alg:eq1};
	
 If $\N(x_3) \in \mathbb{Z}_{\geq 0}$, the bound in Lemma \ref{bd: Nx} is satisfied, and $\N(x_3)$ and $N(x_1 + x_2) = \N(x_1)+\N(x_2)-\Tr(x_1x_2)$ are achievable, compute $\Tr(x_1x_3)$ using \eqref{x_alg:eq3};
 
 If $\N(x_1) = \Tr(x_1x_3) =0 $ or $\N(x_3) = \Tr(x_1x_3) =0 $ or $\N(x_1)\N(x_3) \neq 0$ , compute $\D(x_1x_3)$;
 
 If $\Tr(x_1x_3) \in \mathbb{Z}$, the bound in Lemma \ref{Lemma:bdTrace} holds for $\Tr(x_1x_3)$, $p \mid \D(x_1x_3)$, and $\N(x_1+x_3)=\N(x_1)+\N(x_3)-\Tr(x_1x_3)$ is achievable, and the bounds in Lemma \ref{bd: Nx} are satisfied, compute $\Tr(x_2x_3)$ using \eqref{x_alg:eq2} ;

 If $ \N(x_2) = \Tr(x_2x_3) =0 $ or $\N(x_3) =  \Tr(x_2x_3) =0$ or $\N(x_2)\N(x_3) \neq 0$, compute $\D(x_2x_3)$;
	
 If $\Tr(x_2x_3) \in \mathbb{Z}$, the bound in Lemma \ref{Lemma:bdTrace} holds for $\Tr(x_2x_3)$, $p \mid \D(x_2x_3)$, and $\N(x_2+x_3)=\N(x_2)+\N(x_3)-\Tr(x_2x_3)$ in achievable, compute $\disc(\calR_{123})$; 
 
 If $\disc(\calR_{123})$ is a square and $p \mid \disc(\calR_{123})$;
	
 If $\D(x_1x_2)=\D(x_1x_3)=\D(x_2x_3)=0$ and the solution comes from a solution $d_1, d_2, d_3$ with $p \mid fDn$, or $\D(x_1x_2), \D(x_1x_3), \D(x_2x_3)$ are not all $0$, then we have a solution 
 $$[x,d/n,a,c/n,b,\gamma,n, \N(x_1), \N(x_2), \N(x_3),  \Tr(x_1x_2),  \Tr(x_1x_3),  \Tr(x_2x_3) ].$$

\end{algorithm}

When working with $d_1$, $d_2$, and $d_3$, instead of $x_1$, $x_2$, and $x_3$, we modify the algorithm accordingly. 

\begin{remark} \label{remark: extra_autom}
If $\zeta_4 \in \O$ (or in general if the curve has extra automorphisms), we eliminate solutions with non-zero $d_2$ and $d_3$ by Corollary~\ref{cor: lambda3}.
\end{remark}

\subsubsection{Elements with a given norm} Given a prime $p$, a quaternion order $\calR$ in a quaternion algebra $\mathcal{B}_{p, \infty} = \left( \frac{-q,-p}{\mathbb{Q}} \right)$, where $q \in \mathbb{Z}_{> 0}$, and a certain non-negative integer $N$, we have developed an efficient algorithm to find all elements with norm $N$ that belong to $\calR$. In particular, as described at the beginning of Section \ref{SS:NT}, a quaternion order $\calR$ given by an integral basis $B = M\begin{pmatrix} 1 & i & j & k \end{pmatrix}^T$, where $M = (a_{ij})_{0\leq i,j \leq 3}$ is a $4 \times 4$ matrix with rational entries. Let $C = (c_{ij})_{0\leq i,j \leq 3}$ be the $4 \times 4$ matrix with $c_{ij}$ being the minor obtained by deleting row $i$ and column $j$ of $M$. Let $B_i = (q c_{i0}^2 + c_{i1}^2)p + (qc_{i2}^2 + c_{i3}^2)$, where $0 \leq i \leq 3$. An element $y \in \calR$ can be written as $y = u b_{1} +  v b_{2} + w b_{3} + t b_{4}$, where $b_{r}$ is the $r$-th term in $B$ and $u, v,w,t \in \mathbb{Z}$. Then, if $\N(y) = N$, we have the following bounds
$$
u \leq \frac{\sqrt{|B_0|N/(qp)}}{\det(M)}, \quad v \leq \frac{\sqrt{|B_1|N/(qp)}}{\det(M)}, \quad w \leq \frac{\sqrt{|B_2|N/(qp)}}{\det(M)}, \quad
t \leq \frac{\sqrt{|B_3|N/(qp)}}{\det(M)}.
$$
Instead of writing the algorithm with four nested loops, over $u,v,w,t$, we could use three nested loops and compute the last of the four variables as a solution to the quadratic equation $\N(y) = N$. However, we are only interested in elements $y$ with $\Tr(y) = 0$. Thus, we express one of the variables in terms of the other three, use two nested loops over two of the variables, and compute the last one as a solution to the quadratic equation $\N(y) = N$. This covers Step 5 of \textbf{Algorithm \ref{Malg}}. Then, we use the following.

\RestyleAlgo{ruled}
\LinesNumbered
\begin{algorithm} \caption{\label{sa3}}
 \KwData{$p$ and $[x,d/n,a,c/n,b,\gamma,n, N_1, N_2, N_3,  T_{12},  T_{13}, T_{23}]$, all $y_i \in \calR$ with $\Tr(y_i) = 0$ and $\N(y_i) \in \{N_1, N_2, N_3\}$}
 \KwResult{ All $[x,d/n,a,c/n,b,\gamma,n, x_1,x_2,x_3] $ satisfying equations \eqref{x_alg:eq1} - \eqref{x_alg:eq3}}

For $y_1 \in \calR$ with $\Tr(y_1)=0$ and $\N(y_1)=N_1$;
	
 For $y_2 \in \calR$ with $\Tr(y_2)=0$ and $\N(y_2)=N_2$;
	
 If $\Tr(y_1y_2)=T_{12}$;
	
 For $y_3 \in \calR$ with $\Tr(y_3)=0$ and $\N(y_3)=N_3$;
	
 If $\Tr(y_1y_3)=T_{13}$ and $\Tr(y_2y_3)=T_{23}$;
	
Return all $[x,d/n,a,c/n,b,\gamma,n, y_1, y_2, y_3]$.
\end{algorithm}

\subsubsection{Equivalence}\label{sa4} After computing all the solution with Algorithm \ref{sa3}, we use Propositions \ref{prop: equivforExA}, and \ref{Prop:EquivAutom} to obtain a set of equivalence classes.

We implemented these algorithms in SageMath~\cite{sage}, they can be found in \cite{IKLLMVCode}.

%% file: Sec7-Examples.tex

In Sections~\ref{sec: hyp_ex} and \ref{sec: nonhyp_ex}, we consider hyperelliptic and non-hyperelliptic curves of genus~3 defined over $\Q$ whose endomorphism rings over $\overline{\Q}$ are isomorphic to the maximal order in a sextic CM field. The conjectural equations for these curves are given in \cite{IKLLMMV_mod, KLLRSS} and the correctness of some of these curves is verified by J. Sijsling with the algorithm in \cite{CMSV}. 
We compute the number of solutions to the embedding problems for the maximal orders corresponding to these curves with the algorithm developed in Section~\ref{sec:alg}. Then we compare the primes for which we find solutions and the primes dividing the \emph{minimal} discriminant of the curves. 

Each of these CM fields contains an imaginary quadratic field. We computed the \Int~IEP solutions for the cubic and the quadratic maximal orders and also computed for the sextic maximal order (see Remark~\ref{remark: imaginary}). In both cases, we obtained the same number of  \Int~IEP solutions.

In Section~\ref{sec: nonhyp_ex2}, we consider non-hyperelliptic curve examples with CM by the maximal order in a sextic CM field not containing an imaginary quadratic field. In these cases we do not have curve equations, but we compute the norms of the \emph{minimal} discriminants by computing the theta constants of period matrices with high precision. Then we compare the primes for which we find solutions and the primes dividing the minimal discriminants. 

In all examples, we also experiment with different parametrizations $(A,B,C)$  of the fields and obtained compatible outputs.

\subsection{Hyperelliptic CM curves of genus $3$ defined over the rationals} \label{sec: hyp_ex}

Up to isomorphism there are~8 hyperelliptic curves of genus 3 over $\Q$ with CM  by a maximal order, see \cite[Proposition~4.3.12, Table 3.1]{Kilicerthesis}; the equations for these curves are given in \cite[Section~5.2]{IKLLMMV_mod}. 
By running \textbf{Algorithm \ref{Malg}} for the maximal orders corresponding to these curves, we obtain Table~\ref{table: cubicANDimagfields}. 

\begin{table}[h!]
 \caption{Hyperelliptic curves of genus 3 defined over $\mathbb{Q}$. See \ref{not71} for the explanation of the notation.} \label{table: cubicANDimagfields} 
\centering
\begin{tabular}{|c|c|c|c|c|c|c|}
\hline
Curve & $(A,B,C)$     	& $I$ & $D$  & $D_{14}^{\text{min}}$  & $\# \left\{ \text{\Int~IEP solns.} \right\}$ & $\text{Bd. } p$ \\ 
\hline
$(1)$  & $(13,50,49)$  	& $1$  				& $-1$ 	    & $7^{24}\cdot 11^{12}$ & $7 \cdot 11$ & 11\\ 
\hline
$(2)$  & $(6,9,1)$     		& $1$  				& $-1$ 	    & $3^{8}$ 			 & $3$ & 3 \\ 
\hline
$(3)$  & $(5,6,1)$     		& $1$   				& $-1$ 	    & $1$ 				 & $1$ & 1\\ 
\hline
$(4)$  & $(7,14,7)$    	& $1$  				& $-7$ 	    & $1$ 				 & $1$ & 1 \\ 
\hline
$(5)$  &$(12,27,15)$ 	& $1$  				& $-7$ 	    & $3^8\cdot5^{24}$   & $3^2 \cdot 5^2 \cdot 7^2 $ & 127\\ 
\hline
$(6)$  & \multirow{2}{*}{$(10,19,2)$}  		& \multirow{2}{*}{$2$}   				& \multirow{2}{*}{$-1$} 	    & $2^{?}\cdot11^{24}$ & \multirow{2}{*}{$2^7 \cdot 11^2 \cdot 47$} & \multirow{2}{*}{$47$}\\ 
 $(6)^*$ & 			&  					&   		    & $2^?\cdot 11^{12}\cdot 47^{24}$ & & \\ 
 \hline
$(7)$ & \multirow{2}{*}{$(11,30,16)$} 	& \multirow{2}{*}{$2$} 				& \multirow{2}{*}{$-1$} 	    & $2^?$ 			& \multirow{2}{*}{$2^5 \cdot 23$} & \multirow{2}{*}{$23$}\\ 
$(7)^*$ & 				& 	 	    			& 		    & $2^?\cdot 23^{24}$ 	& & \\ 
\hline
$(8)$  & \multirow{2}{*}{$(12,27,8)$}  		& \multirow{2}{*}{$2$}  				& \multirow{2}{*}{$-1$} 	    & $2^?\cdot3^8 $  		& \multirow{2}{*}{$2^{11} \cdot 3^4 \cdot 7^2 \cdot 31 \cdot 47 \cdot 79$} & \multirow{2}{*}{$ 79$} \\ 
$(8)^*$  &  			&   					&  	    	    & $2^?\cdot3^?\cdot 7^?\cdot 31^{12}\cdot 47^{24}\cdot 79^{24}$  & & \\ 
\hline
\end{tabular}
\end{table}
\subsubsection{Notation of  Table \ref{table: cubicANDimagfields}}\label{not71}
\begin{itemize}
\item The numbering of the hyperelliptic curves corresponds to the CM fields with the same numbering in \cite{IKLLMMV_mod}.  
\item Each of these sextic CM fields is the compositum of a totally real cubic field $K_+$ generated by the polynomial $t^3 - At^2 + Bt -C$ and an imaginary quadratic field $k=\Q(\sqrt{D})$, where $D\in \Z_{<0}$ is square-free. Note that the coefficients in the second column do not necessarily match the ones in \cite{IKLLMMV_mod}. 
\item  We define  $I := [\O_{K_+} : \Z[\mu]]$.
\item  CM fields (6)--(8) have $h_K/h_{K_+} = 4$. Hence, by \cite[Theorem 4.3.1]{Kilicerthesis}, there are $4$ curves with CM by the maximal orders of these CM fields; one is defined over $\Q$ and the other three are defined over $K_+$. The curves defined over $K_+$ are Galois conjugates and represented by the cases marked with $^*$.
\item The entries in the column $D_{14}^{\text{min}}$ are the norms  of the minimal discriminant of the curves:
By minimal discriminant we understand the product of primes appearing in the discriminant of any model with the normalized exponent using a Homogeneous System Of Parameter (HSOP) as in \cite[Definition~1.6]{LLLR}.
Shioda invariants form a HSOP for hyperelliptic curves for primes greater than $7$ and for small characteristic we used the HSOPs proposed in the literature~\cite{LRS23,LRS-CodeH}. For $p=2$ an HSOP is not known, so we could not compute the exact valuation in some cases and we wrote ``$?$" for it. 
\item The exponents of the primes in the column ``\#\{\Int~IEP solns\}" correspond to the number of embeddings up to equivalence in a multiplicative way for the primes appearing: $p^e$ means that there are $e$ of \Int~IEP solutions for the prime $p$.
\item The last column gives a bound for primes for which we run \textbf{Algorithm \ref{Malg}}. The bounds are given in Theorems \ref{thm: PrimeBdQ(i)} and \ref{thm: PrimeBdImagQuad}.
\end{itemize}

\subsubsection{Interpretation of results in Table \ref{table: cubicANDimagfields}}

\begin{itemize}
\item All curves except (5) have extra automorphisms; in cases (1)--(3) and (6)--(8) we have $i\in \O_K$ and in case (4) we have $\zeta_7 \in \O_K$. For these cases we eliminate solutions also with non-zero $d_2$ and $d_3$, see Remark~\ref{remark: extra_autom}.
\item When a prime $p$ divides $[\O_{K_+} : \Z[\mu]]$, we observe that we sometimes obtain extra solutions with~$p$ dividing $n$, so we cannot a priori rule out these solutions because we do not know how to detect the primitivity of the CM type modulo $p$ in this situation, see Proposition \ref{prop: primitivity}. By working with a different $(A,B,C)$, we usually get to discard them by finding an equivalent solution that can be discarded. 
\item We observe the same phenomenon with primes $p$ dividing $D$, for example, for case $(5)$ in Table~\ref{table: cubicANDimagfields} there are two unexpected solutions at $7$. This seems to be related with the fact that $7$ divides $D$ in this case, and we cannot control the CM type in this situation, see Proposition \ref{prop: primitivity}. Except for this case our algorithm outputs solutions exactly for the primes of bad reduction. 

\item The exponents in the sixth column of Table~\ref{table: cubicANDimagfields} are not related to the exponents in the discriminant.
The meaning of the exponents of the last columns are related with the number of different reductions of curves with CM by the corresponding order at a prime $p$. This only explains the exponent $1$ at $7$ and $11$ for case (1): there is only one curve with CM by the maximal order of the first field and with bad reduction at each prime. 
The exponent $2$ at $p=11$ in case (6) makes reference to the fact that there are (at most) two types of bad reduction at $11$: the one of the curve $(6)$ and the one of the three conjugated curves in case~$(6)^*$. \\
In case (7), there are (at most) four types of bad reduction at $2$: the one of the curve $(7)$ and the three of the three conjugated curves in case $(7)^*$. Since $2$ totally split in $K_+$ the reduction of these three curves at a prime dividing $2$ are a priori non-isomorphic. However, we obtain an exponent of~$5$, which means that we may be overcounting when $p$ divides the index $[\O_{K_+} : \Z[\mu]]$. The same thing may be happening for the prime $2$ in cases (6) and (8). \\
Also in case (5) we get some overcounting at $p=5$, but the reason here seems to be that $p\mid n$ in one of the solution we find, we suspect this extra solution corresponds to non-primitive CM-type.
\end{itemize}

\subsection{Non-hyperelliptic CM curves of genus $3$ defined over the rationals}\label{sec: nonhyp_ex} 

Up to isomorphism there are~$29$ isomorphism classes of non-hyperelliptic curves of genus $3$, i.e., plane quartics, defined over $\Q$ with CM by a maximal order, \cite[Proposition~4.3.12, Table 3.1]{Kilicerthesis}; the equations for the corresponding curves are given in \cite[Section~5]{KLLRSS}.
In this section, we consider some interesting cases of these plane quartic curves, namely $X_1$, $X_3$, $X_5$, $X_6$, $X_7$, $X_8$, $X_9$, $X_{10}$, $X_{11}$, $X_{12}$ and $X_{17}$ numbered as in \cite{KLLRSS}. For non-hyperelliptic curves of genus $3$, the primes dividing the discriminant are either primes of potential good reduction, bad reduction or hyperelliptic reduction.  In \cite{LLLR}, the reduction types for plane quartics are listed, however, in some cases, the criteria given in the paper cannot be used to decide the reduction types modulo $2$, $3$ or $7$. We will look into these cases in this section.

\begin{landscape}
\begin{table}[htp]
 \caption{Some non-hyperelliptic curves of genus 3 defined over $\mathbb{Q}$.  See \ref{not72} for the explanation of the notation.} \label{table: planequarticfields} 
\centering 
\begin{tabular}{|c|c|c|c|c|c|c|c|}
\hline
Curve & $(A,B,C)$     	& $I$    & $D$   & $I_{27}^{\text{min}} $ & $\# \left\{ \text{\Int~IEP solns.} \right\}$ & Bd. $p$ & Ran up to\\ 
\hline
 $X_1$   & $(7,12,5)$ 	& $1$ 				& $-7$    & 	$-2^{70}\cdot 5^{12}\cdot 7^9 \cdot  37^{14} \cdot 15187^{14}$	 & $5 \cdot 7^2$ & 2484	 & 2484 \\
\hline
\multirow{2}{*}{$X_3$}    & \multirow{2}{*}{$(20,61,15)$}	& \multirow{2}{*}{$3$}                   		& \multirow{2}{*}{$-7$}    & \multirow{2}{*}{$2^{42}\cdot3^{18}\cdot5^{36}\cdot7^{7}\cdot233^{14}\cdot356399^{14}$} & \multirow{2}{*}{$3^{10} \cdot 5^4 \cdot 7^6 \cdot 17 \cdot ?$} & \multirow{2}{*}{68266800}	&  
$22000$ \& \\
&&&&&&& $p = 356399$\\  
\hline
$X_5$    & $(16,55,27)$	& $3$                    		& $-7$    & $2^{42}\cdot3^{24}\cdot7^{7}\cdot37^{14}\cdot127^{14}$ & $3^2 \cdot 7^2 \cdot ?$ & 4105728	 &  1000000		\\ 
\hline
$X_6$    & $(19,76,57)$	& $3$                    		& $-7$    & $2^{42}\cdot3^{24}\cdot7^{7}\cdot17^{12}\cdot127^{14}\cdot211^{14}\cdot20707^{14}$ & $3^4 \cdot 7^4 \cdot 17 \cdot ?$& 19411840	 & 22000		\\ 
\hline 
$X_7$    & $(16,61,3)$	& $3$                    		& $-7$    & $- 2^{42} \cdot 3^{18} \cdot 7^7 \cdot 71^{14} \cdot 83^{12} \cdot 17665559^{14}$ & $3^{17} \cdot 7^{12}  \cdot 17^7 \cdot 83 \cdot 167 $& 1096200	 &  1096200 \\ 
\hline 
$X_8$    & $(8,15,7)$	& $1$                    		& $-7$    & $2^{46} \cdot 7^{15} \cdot 499^{14}$ & $7^3$& 10800	 & 10800 		\\ 
\hline  
$X_9$    & $(7,12,5)$	& $1$                    		& $-2$    & $-2^{45}\cdot5^{12}\cdot7^{14}\cdot79^{14}\cdot233^{14}\cdot857^{14}$  &  $2^2 \cdot 5$ & 1680 &  1680 \\ 
\hline
$X_{10}$ & $(5,6,1)$		& $1$                    		& $-2$    & $ -2^{45}\cdot7^{14}\cdot41^{14}\cdot71^{14}$ & $2$ & 160	 &  160 \\ 
\hline
$X_{11}$ & $(13,46,32)$		& $2$                    		& $-2$    & $ - 2^{85} \cdot 7^{14} \cdot 23^{14} \cdot 47^{14} \cdot 27527^{14}$ & $2^5 \cdot 23^7$ & 28480	 & 28480 \\ 
\hline
$X_{12}$ & $(5,6,1)$		& $1$                    		& $-11$  & $7^{14}\cdot11^{9}\cdot5711^{14}\cdot73064203493^{14}$ & $11$ & 160 & 160 \\ 
\hline
$X_{17}$ & $(18,51,26)$	& $2$                    		& $-19$  & $2^{24}\cdot3^{50}\cdot19^{7}\cdot1229^{14}\cdot3913841117^{14}$  & $2^5 \cdot 3^{10} \cdot 19^2 \cdot ?$& 32292864	 &  6000	\\ 
\hline  
 \end{tabular}
 \end{table}

\subsubsection{Notation of  Table \ref{table: planequarticfields}}\label{not72}
\begin{itemize}
\item The numbering and notation for the non-hyperelliptic curves $X_j$ and their CM fields as in~\cite{KLLRSS}.
\item The columns $(A,B,C)$, $I$, $D$ and $\#\{\Int~\text{IEP sols.}\}$ are defined exactly as in Table~\ref{table: cubicANDimagfields}. Note that the coefficients in the second column do not necessarily match the ones in \cite{KLLRSS}. 
\item $I_{27}^\text{min}$ is the minimal discriminant. In this case, by minimal discriminant we understand the product of primes appearing in the discriminant of any model with the normalized exponent using an HSOP as in \cite[Definition~1.6]{LLLR}. Note that in \cite{KLLRSS} this notation is used with a different meaning. Dixmier-Ohno invariants contain an HSOP for primes $>7$, see \cite[Theorem~4.1]{LLLR}. For smaller primes we need to look for different HSOP's, see \cite{LRS23,LRS-CodeQ}. 
\item The column ``Bd. $p$" is the bound for the primes of bad reduction as given in Remark~\ref{remark: PrimeBdgen}. 
The last column gives a bound for primes for which we run \textbf{Algorithm \ref{Malg}}. In the cases that we did not run \textbf{Algorithm \ref{Malg}} for each prime up to the bound in ``Bd. $p$", we write $?$ in the number of solutions. Note that even if we do not always run the algorithm up to the bounds in ``Bd. $p$", we check all primes dividing the discriminant $I_{27}^\text{min}$.
\end{itemize}

\end{landscape}

\subsubsection{Interpretation of results in Table \ref{table: planequarticfields}}
\begin{itemize}
\item The CM fields corresponding to the rational curves $X_7$, $X_8$, $X_{11}$ have relative class number 4. Hence, by \cite[Theorem 4.3.1]{Kilicerthesis} for each CM field $K$ there are four curves of genus 3 with the maximal order $\OO_K$; one is defined over $\Q$ and the other three are defined over $K_+$. The curves defined over $K_+$ are Galois conjugates. As we can see in case $X_7$ there are \Int~IEP solutions for primes $17$ and $167$ even though the minimal discriminant of $X_7$ is not divisible by these primes. One explanation for this could be that $17$ and~$167$ are primes of bad reductions for the conjugate curves defined over $K_+$. However, this would not explain the solution for $p=17$ for the CM field of curve $X_3$: of course, it could happen that there are solutions not providing CM curves and bad reduction, but our guess here is that there is a CM curve with CM by the field but with non-primitive CM-type and bad reduction at $17$. Unfortunately, at the moment nobody have computed curves equations for non-primitive sextic CM fields, so we cannot check this hypothesis.

\item The index $[\O_{K_+} : \Z[\mu]]$ is divisible by 3 in cases $X_3$, $X_5$, $X_6$ and $X_7$, which may lead to extra \Int~IEP solutions for prime $3$. As stated in \cite[Table 4]{LLLR}, prime $3$ is a prime of bad reduction in each of these cases. Thus, at least some of the solutions we find are really coming from bad reduction.
\item The prime 7 divides $D$ for in case $X_1$, which may lead to extra \Int~IEP solutions, as in case (5) in Table \ref{table: cubicANDimagfields}. As stated in \cite[Table 4]{LLLR}, prime $7$ is a prime of bad reduction for $X_1$. Thus, at least some of the solutions we find are really coming from bad reduction.
\item The criteria in~\cite{LLLR} cannot be used to decide the reduction type of $X_{17}$ at $3$. Since we found a solution at $3$ and $3$ does not divide $D$ or $[\O_{K_+} : \Z[\mu]]$, we suspect that $3$ is indeed a prime of bad reduction.
\item As shown in~\cite[Table 4]{LLLR}, all primes appearing in the minimal discriminant
with exponent $14$ are primes of hyperelliptic reduction. We did not find solutions to the embedding problem for any of these primes, which confirms that that they are not primes of bad reduction. 
\end{itemize}

\begin{proposition} \label{prop:hyperred}
The curves $X_{9}, X_{11}, X_{10}$ and $X_{12}$ have potentially good hyperelliptic reduction at $p=7$.
\end{proposition}
\begin{proof} The algorithm does not output any solution for the corresponding fields at $p=7$ and bad reduction would produce solutions according to Proposition \ref{prop: embeddingI}.
\end{proof}

\subsection{Examples of sextic CM field examples with no imaginary quadratic subfield} \label{sec: nonhyp_ex2}

In this section, we will be interested in the maximal orders of sextic CM fields with no imaginary quadratic subfield and in the curves with CM by these maximal orders. We list some examples of such fields in Table \ref{table: EmbeddingsGeneric}; all of these fields have class number one. 

For each CM field $K$ in Table \ref{table: EmbeddingsGeneric}
there are four CM curves of genus 3, up to isomorphism and they are defined over the totally real subfields of the corresponding reflex fields.

Numerical computation with high precision shows that the even theta constants of these curves are non-zero. Therefore, by \cite{igusa67}, we conclude that these curves are plane quartics.

\begin{table}[htp] 
\caption{Some sextic CM fields not containing an imaginary quadratic field. See \ref{not73} for the explanation of the notation.}  \label{table: EmbeddingsGeneric} 
\centering
\begin{tabular}{|c|c|c|c|c|c|c|c|c|}
\hline
Curve  & $(A,B,C)$      & $I$   & $I^\text{min}_{27}$ & $\# \left\{ \text{\Int~IEP sols.} \right\}$ & Bd. $p$ & Ran up to\\ \hline
$G_1$  & $(15,14,3)$   & $1$   & $\mathfrak{p}_9^{18} \mathfrak{p}_5^3 \mathfrak{p}_{1187}^{14}$  & $2^3 \cdot 5 \cdot ?$ &23080960 & 1367 \\ \hline 
$G_2$  & $(11,34,31)$  & $1$  & $\mathfrak{p}_8^{20} \mathfrak{p}_3^{18}\mathfrak{p}_{27}^{32} \mathfrak{p}_{1373}^{14}$ &$3^2$ & 31680 & 31680 \\ \hline 
$G_3$  & $(7,10,2)$     & $1$   & $\mathfrak{p}_2^{32}\mathfrak{p}_{3}^{18}\mathfrak{p}_7^{12}$    & $2^2\cdot 7$ &6660 &6660 \\ \hline 
$G_4$     & $(9,21,8)$     & $1$   & $\mathfrak{p}_3^{18}$ & 1 &7920 &7920\\ \hline 
$G_5$     &$(10,29,23)$  & $1$   & $\mathfrak{p}_8^{20}\mathfrak{p}_3^{18}\mathfrak{p}^9_9\mathfrak{p}_{47}^{14} \mathfrak{p}_{73}^{14}$      & 1&8208 &8208 \\ \hline 
 \end{tabular}
\end{table}

\subsubsection{Notation of  Table \ref{table: EmbeddingsGeneric}}\label{not73}
\begin{itemize}
\item Each of these sextic CM fields is generated by the polynomial $t^6 + At^4 + Bt^2 + C$.

\item The columns $(A,B,C)$, $I$, $D$, $\#\{\Int~\text{IEP sols.}\}$ and Bd. $p$ are defined as in Table~\ref{table: cubicANDimagfields}. 

\item The columns $I_{27}^\text{min}$ and ``Ran up to" are defined as in Table~\ref{table: planequarticfields}. The notation $\mathfrak{p}_m$ denotes a prime ideal of norm $m$.

\end{itemize}

As we discussed earlier, we can read the following facts from Table~\ref{table: EmbeddingsGeneric}.
The curves corresponding to $G_1$ and $G_3$ respectively have potentially good hyperelliptic reduction at $p=1187$ and $3$ respectively. 
The curves corresponding to $G_2$ have potentially good hyperelliptic reduction at $2$ and $1373$. The curves corresponding to $G_4$ and $G_5$ have potentially good reduction everywhere.

%% file: Paper.bbl
\begin{thebibliography}{10}

\bibitem{AEPicard}
S.~Arora and K.~Eisentr\"{a}ger.
\newblock Constructing {P}icard curves with complex multiplication using the
  {C}hinese remainder theorem.
\newblock In {\em Proceedings of the {T}hirteenth {A}lgorithmic {N}umber
  {T}heory {S}ymposium}, volume~2 of {\em Open Book Ser.}, pages 21--36. Math.
  Sci. Publ., Berkeley, CA, 2019.

\bibitem{BILV}
J.~S. Balakrishnan, S.~Ionica, K.~Lauter, and C.~Vincent.
\newblock Constructing genus-3 hyperelliptic {J}acobians with {CM}.
\newblock {\em LMS J. Comput. Math.}, 19(suppl. A):283--300, 2016.

\bibitem{Barretto}
P.~S. L.~M. Barreto, B.~Lynn, , and M.~Scott.
\newblock Constructing elliptic curves with prescribed embedding degrees.
\newblock In {\em Security in Communication Networks — Third International
  Conference, SCN 2002}, Lecture Notes in Computer Science, vol. 2576,
  Springer-Verlag, page 257–267, 2003.

\bibitem{DIS}
B.Dina, S.~Ionica, and J.~Sijsling.
\newblock Isogenous hyperelliptic and non-hyperelliptic {Jacobians} with
  maximal complex multiplication.
\newblock {\em Mathematics of {Computation}}, 92:349--383, 2023.

\bibitem{BBEL}
J.~Belding, R.~Broker, A.~Enge, and K.~Lauter.
\newblock Computing {H}ilbert class polynomials.
\newblock In {\em Algorithmic number theory}, volume 5011 of {\em Lecture Notes
  in Comput. Sci.}, pages 282--295. Springer, Berlin, 2008.

\bibitem{BCLLMNO}
I.~Bouw, J.~Cooley, K.~E. Lauter, E.~Lorenzo~Garc{\'\i}a, M.~Manes, R.~Newton,
  and E.~Ozman.
\newblock Bad reduction of genus $3$ curves with complex multiplication.
\newblock In {\em Women in Numbers Europe}, volume~2 of {\em Association of
  Women in Mathematics}, pages pp. 57--82, 2015.

\bibitem{BrezingWeng}
F.~Brezing and A.~Weng.
\newblock Elliptic curves suitable for pairing based cryptography.
\newblock {\em Designs, Codes and Cryptography}, 37(1):133–141, 2005.

\bibitem{Broker}
R.~Broker.
\newblock A {$p$}-adic algorithm to compute the {H}ilbert class polynomial.
\newblock {\em Math. Comp.}, 77(264):2417--2435, 2008.

\bibitem{BruinierYang}
J.~H. Bruinier and T.~Yang.
\newblock C{M}-values of {H}ilbert modular functions.
\newblock {\em Invent. Math.}, 163(2):229--288, 2006.

\bibitem{ChaiConradOort}
C.-L. Chai, B.~Conrad, and F.~Oort.
\newblock {\em Complex multiplication and lifting problems}, volume 195 of {\em
  Mathematical Surveys and Monographs}.
\newblock American Mathematical Society, Providence, RI, 2014.

\bibitem{CMSV}
E.~Costa, N.~Mascot, J.~Sijsling, and J.~Voight.
\newblock Rigorous computation of the endomorphism ring of a {J}acobian.
\newblock 2017.

\bibitem{DI}
B.~Dina and S.~Ionica.
\newblock Genus 3 hyperelliptic curves with cm via shimura reciprocity.
\newblock In {\em Proceedings of the {Fourteenth Number Theory Symposium}},
  volume~4 of {\em Open Book Series}, 2020.

\bibitem{Morain}
R.~Dupont, A.~Enge, and F.~Morain.
\newblock Building curves with arbitrary small mov degree over finite prime
  fields.
\newblock {\em Journal of Cryptology}, 18(2):79–89, 2005.

\bibitem{GorenLauter07}
E.~Z. Goren and K.~E. Lauter.
\newblock Class invariants for quartic {CM} fields.
\newblock {\em Ann. Inst. Fourier (Grenoble)}, 57(2):457--480, 2007.

\bibitem{GL11}
E.~Z. Goren and K.~E. Lauter.
\newblock Genus 2 curves with complex multiplication.
\newblock {\em Int. Math. Res. Notices}, 2011.
\newblock Published online April 12, 2011, doi:10.1093/imrn/rnr052.

\bibitem{GrossZagier}
B.~H. Gross and D.~B. Zagier.
\newblock On singular moduli.
\newblock {\em J. Reine Angew. Math.}, 355:191--220, 1985.

\bibitem{igusa67}
J.-i. Igusa.
\newblock Modular forms and projective invariants.
\newblock {\em Amer. J. Math.}, 89:817--855, 1967.

\bibitem{IKLLMMV_mod}
S.~Ionica, P.~K{\i}l{\i}\c{c}er, K.~Lauter, E.~Lorenzo~Garc\'{\i}a,
  A.~M\^{a}nz\u{a}\c{t}eanu, M.~Massierer, and C.~Vincent.
\newblock Modular invariants for genus 3 hyperelliptic curves.
\newblock {\em Res. Number Theory}, 5(1):Paper No. 9, 22, 2019.

\bibitem{IKLLMVCode}
S.~Ionica, P.~K{\i}l{\i}\c{c}er, K.~Lauter, E.~Lorenzo~Garc\'{\i}a,
  A.~M\^{a}nz\u{a}\c{t}eanu, M.~Massierer, and C.~Vincent.
\newblock The embedding problem.
\newblock \url{https://github.com/AdelinaManzateanu/the-embedding-problem},
  2021.

\bibitem{Kani}
E.~Kani.
\newblock Products of {CM} elliptic curves.
\newblock {\em Collect. Math.}, 62(3):297--339, 2011.

\bibitem{Kilicerthesis}
P.~K{\i}l{\i}\c{c}er.
\newblock {\em The {CM} class number one problem for curves}.
\newblock PhD thesis, Leiden University and University of Bordeaux, 2016.

\bibitem{KLLRSS}
P.~K\i{}l\i{}\c{c}er, H.~Labrande, R.~Lercier, C.~Ritzenthaler, J.~Sijsling,
  and M.~Streng.
\newblock Plane quartics over {$\Bbb{Q}$} with complex multiplication.
\newblock {\em Acta Arith.}, 185(2):127--156, 2018.

\bibitem{KLLNOS}
P.~K{\i}l{\i}\c{c}er, K.~E. Lauter, E.~Lorenzo~Garc\'ia, R.~Newton, E.~Ozman,
  and M.~Streng.
\newblock A bound on the primes of bad reduction for {CM} curves of genus 3.
\newblock {\em Proc. Amer. Math. Soc.}, 148(7):2843--2861, 2020.

\bibitem{KLS}
P.~K{\i}l{\i}\c{c}er, E.~Lorenzo~Garc\'ia, and M.~Streng.
\newblock Primes dividing invariants of cm picard curves.
\newblock {\em Canadian Journal of Mathematics}, 72(2):480--504, 2020.

\bibitem{KoikeWeng}
K.~Koike and A.~Weng.
\newblock Construction of {CM} {P}icard curves.
\newblock {\em Math. Comp.}, 74(249):499--518 (electronic), 2005.

\bibitem{LarioSomoza}
J.-C. Lario, A.~Somoza, and C.~Vincent.
\newblock An inverse {J}acobian algorithm for {P}icard curves.
\newblock {\em Res. Number Theory}, 7(2):Paper No. 32, 23, 2021.

\bibitem{LV15b}
K.~Lauter and B.~Viray.
\newblock An arithmetic intersection formula for denominators of {I}gusa class
  polynomials.
\newblock {\em Amer. J. Math.}, 137(2):497--533, 2015.

\bibitem{LV15a}
K.~Lauter and B.~Viray.
\newblock On singular moduli for arbitrary discriminants.
\newblock {\em Amer. J. Math.}, 19:9206--9250, 2015.

\bibitem{LLLR}
R.~Lercier, Q.~Liu, E.~Lorenzo~Garc\'{\i}a, and C.~Ritzenthaler.
\newblock Reduction type of smooth plane quartics.
\newblock {\em Algebra Number Theory}, 15(6):1429--1468, 2021.

\bibitem{LR19}
R.~Lercier and C.~Ritzenthaler.
\newblock Siegel modular forms of degree three and invariants of ternary
  quartics.
\newblock {\em to appear in the Proceedings of the American Mathematical
  Society}, 2019.

\bibitem{LRS-CodeH}
R.~Lercier, C.~Ritzenthaler, and J.~Sijsling.
\newblock {\tt hyperelliptic}, a \textsc{Magma} repository for reconstruction
  and isomorphisms of hyperelliptic curves.
\newblock \href{https://github.com/JRSijsling/hyperelliptic}{\tt
  https://github.com/JRSijsling/hyperelliptic}, 2020.

\bibitem{LRS-CodeQ}
R.~Lercier, C.~Ritzenthaler, and J.~Sijsling.
\newblock {\tt quartic}, a \textsc{Magma} package for calculating with smooth
  plane quartic curves.
\newblock \href{https://github.com/JRSijsling/quartic}{\tt
  https://github.com/JRSijsling/quar\-tic}, 2020.

\bibitem{LRS23}
R.~Lercier, C.~Ritzenthaler, and J.~Sijsling.
\newblock Functionalities for genus 2 and genus 3 curves.
\newblock {\em to appear in the Proceedings Volume of the Conference MEGA
  2021}, 2021.

\bibitem{Lor20}
E.~Lorenzo~Garc\'ia.
\newblock On different expressions for invariants of hyperelliptic curves of
  genus 3.
\newblock {\em J. Math. Soc. Japan}, 74(2):403--426, 2022.

\bibitem{Oort1974}
F.~Oort.
\newblock Subvarieties of moduli spaces.
\newblock {\em Inventiones Mathematicae}, 24:95--119, 1974.

\bibitem{Oort1975}
F.~Oort.
\newblock Which abelian surfaces are products of elliptic curves?
\newblock {\em Mathematische Annalen}, 214:35--47, 1975.

\bibitem{SerreTate}
J.-P. Serre and J.~Tate.
\newblock Good reduction of abelian varieties.
\newblock {\em Ann. of Math. (2)}, 88:492--517, 1968.

\bibitem{Shimura-Taniyama}
G.~Shimura and Y.~Taniyama.
\newblock {\em Complex multiplication of abelian varieties and its applications
  to number theory}, volume~6 of {\em Publications of the Mathematical Society
  of Japan}.
\newblock The Mathematical Society of Japan, Tokyo, 1961.

\bibitem{Somozathesis}
A.~Somoza.
\newblock {\em Inverse Jacobian and related topics for certain superelliptic
  curves}.
\newblock PhD thesis, Leiden University and University of Bordeaux, 2019.

\bibitem{Streng2014}
M.~Streng.
\newblock Computing {I}gusa class polynomials.
\newblock {\em Math. Comp.}, 83(285):275--309, 2014.

\bibitem{DrewHilbert}
A.~V. Sutherland.
\newblock Computing {H}ilbert class polynomials with the {C}hinese remainder
  theorem.
\newblock {\em Math. Comp.}, 80(273):501--538, 2011.

\bibitem{sage}
{The Sage Developers}.
\newblock {\em {S}ageMath, the {S}age {M}athematics {S}oftware {S}ystem
  ({V}ersion 7.4)}, 2016.
\newblock {\tt http://www.sagemath.org}.

\bibitem{Voight}
J.~Voight.
\newblock {\em Quaternion algebras}, volume 288 of {\em Graduate Texts in
  Mathematics}.
\newblock Springer, Cham, [2021] \copyright 2021.

\bibitem{Weng}
A.~Weng.
\newblock A class of hyperelliptic {CM}-curves of genus three.
\newblock {\em J. Ramanujan Math. Soc.}, 16(4):339--372, 2001.

\bibitem{Zaytsev}
A.~Zaytsev.
\newblock Generalization of {D}euring reduction theorem.
\newblock {\em J. Algebra}, 392:97--114, 2013.

\end{thebibliography}
